\documentclass[11pt]{article}
\usepackage{graphicx}
\usepackage[a4paper,margin=0.6in]{geometry}
\usepackage{float}
\usepackage{empheq}
\usepackage{pdflscape}
\usepackage{multirow}
\usepackage{circledsteps}
\usepackage{enumitem}
\usepackage{comment}
\usepackage[colorlinks=true,citecolor=magenta,linkcolor=blue]{hyperref}
\pdfpagewidth=\paperwidth \pdfpageheight=\paperheight
\usepackage{amsfonts,amssymb,amsthm,amsmath,eucal,tabu,url}
\usepackage{pgf}
\usepackage{caption}
 \usepackage{array}
 \usepackage{booktabs}
 \usepackage{arydshln}
 \usepackage[table]{xcolor}
 \usepackage{makecell}
 \usepackage{tikz-cd}
 \usepackage{pstricks}
 \usepackage{pstricks-add}
 \usepackage{pgf,tikz, pgfplots}
\pgfplotsset{compat=1.18}
 \usetikzlibrary{automata}
 \usetikzlibrary{arrows}
\usetikzlibrary{intersections,patterns,pgfplots.groupplots}
 \usepackage{indentfirst}
\usepackage{tabularx} 

\theoremstyle{plain}
\newtheorem{thm}{Theorem}[section]

\newtheorem{lemma}[thm]{Lemma}
\newtheorem{proposition}[thm]{Proposition}
\newtheorem{corollary}[thm]{Corollary}

\theoremstyle{definition}
\newtheorem{definition}[thm]{Definition}
\newtheorem{remark}[thm]{Remark}
\newtheorem{example}[thm]{Example}

\newtheorem{problem}[thm]{Problem}

\newtheorem{thevarthm}[thm]{\varthmname}

\newenvironment{varthm*}[1]{\trivlist\item[]{\bf #1.}\it}{\endtrivlist}

\def\keywordname{{\bfseries Keywords}}%
\def\keywords#1{\par\addvspace\medskipamount{\rightskip=0pt plus1cm
\def\and{\ifhmode\unskip\nobreak\fi\ $\cdot$
}\noindent\keywordname\enspace\ignorespaces#1\par}}
\def\subclassname{{\bfseries Mathematics Subject Classification
(2020)}\enspace}
\def\subclass#1{\par\addvspace\medskipamount{\rightskip=0pt plus1cm
\def\and{\ifhmode\unskip\nobreak\fi\ $\cdot$
}\noindent\subclassname\ignorespaces#1\par}}

\renewcommand{\r}{\rightarrow}

\let\oldeqref\eqref
\renewcommand{\eqref}[1]{\begingroup\hypersetup{linkcolor=black}\oldeqref{#1}\endgroup}

\begin{document}
\title{On homological properties of conic-line arrangements with simple singularities}
\author{Artur Bromboszcz}
\date{\today}
\maketitle
\thispagestyle{empty}

\begin{abstract}
We study arrangements of smooth conics and lines in the complex projective plane whose singularities are limited to nodes, tacnodes, and ordinary triple points. The first part of the paper gives numerical restrictions for plus-one generated conic arrangements with defect $\nu(C)=3$ and explains how these restrictions interact with B\'ezout's theorem, the Dimca--Sernesi bound for the minimal degree of a Jacobian syzygy, and Hirzebruch-type inequalities. In particular, the possible numbers of conics are bounded, and the exceptional low-degree cases are separated from those that remain open. The second part concerns arrangements of total degree at most $6$. We identify the weak and strong Ziegler pairs occurring in the database recorded in the Appendix. 
\keywords{plus-one generated curves, conic-line arrangements, plane curve singularities, Ziegler pairs.}
\subclass{14N20, 14H50, 32S25.}
\end{abstract}

\section{Introduction}

The topology and homological algebra of plane curve arrangements are strongly influenced by the way in which the irreducible components meet. For line arrangements this relationship has been studied extensively, but arrangements containing higher-degree components display additional phenomena. Already for arrangements of smooth conics, tangencies and higher multiple points may produce the same elementary numerical data while leading to different Jacobian syzygies and different Milnor algebras.

In this paper we work over $\mathbb P^2_{\mathbb C}$ and consider reduced arrangements whose components are smooth conics and lines. We restrict the singularities to three simple types: nodes, tacnodes, and ordinary triple points. This restriction is narrow enough to make a finite low-degree classification feasible, but broad enough to include examples where the combinatorics does not determine the homological type.

Our first goal is to understand plus-one generated conic arrangements with defect $\nu(C)=3$. We combine the Dimca--Sernesi lower bound for the minimal degree of a Jacobian syzygy with B\'ezout's theorem and, when needed, a Hirzebruch-type inequality for conic arrangements. This gives strong numerical restrictions: if an arrangement of $k$ smooth conics is plus-one generated with defect $3$, then $k\leq 6$, and the possible weak combinatorics are sharply constrained. We also exhibit an explicit arrangement of four conics realizing the weak combinatorics $(n_2,t_3,n_3)=(3,9,1)$. This example is included to separate two issues which should not be confused: geometric realizability of a weak combinatorics and existence of a realization with the desired homological type.

The second goal is to compare weak and strong combinatorial data with homological data. We use a database of conic-line arrangements of degree at most $6$, recorded in the appendix, and compute the minimal graded free resolutions of the corresponding Milnor algebras. This produces weak Ziegler pairs, where the weak combinatorics agree but the resolutions differ, and strong Ziegler pairs, where even the stronger incidence data agree while the Jacobian syzygy modules are not isomorphic.

The paper is organized as follows. Section~\ref{sec:POG3} studies plus-one generated arrangements of conics with defect $\nu(C)=3$. Section~\ref{sec:CL} records a geometric restriction for arrangements consisting of two conics and one line; this explains why some weak combinatorics cannot be realized. Section~\ref{Ziegler} contains the classification of weak and strong Ziegler pairs among the arrangements in the database. The appendix contains the database itself and is used as the finite source of the examples appearing in the final section. 

All necessary symbolic computations were performed using \texttt{SINGULAR} \cite{Singular}.

\section{Preliminaries}
We recall the notation and the numerical tools used throughout the paper. Let $C:f=0$ be a reduced curve in $\mathbb{P}^{2}_{\mathbb{C}}$, where $f\in S:=\mathbb{C}[x,y,z]$ is homogeneous of degree $d$. We write
\[
J_f=\langle \partial_x f,\partial_y f,\partial_z f\rangle
\]
for the Jacobian ideal and
\[
M(f)=S/J_f
\]
for the Milnor algebra of $C$. The module of Jacobian syzygies is
\[
{\rm AR}(f)=\{(a,b,c)\in S^{\oplus 3}:a\partial_xf+b\partial_yf+c\partial_zf=0\},
\]
and its first non-zero degree is
\[
{\rm mdr}(f)=\min\{r\geq0:{\rm AR}(f)_r\neq0\}.
\]

\begin{definition}
We say that $C$ is an $s$-syzygy curve if the Milnor algebra $M(f)$ has a minimal graded free resolution of the form
\[
0 \rightarrow \bigoplus_{i=1}^{s-2}S(-e_i) \rightarrow \bigoplus_{j=1}^{s}S(1-d-d_j) \rightarrow S^3(1-d)\rightarrow S\rightarrow M(f)\rightarrow0,
\]
where $e_1\leq\cdots\leq e_{s-2}$ and $1\leq d_1\leq\cdots\leq d_s$. The tuple $(d_1,\ldots,d_s)$ is called the tuple of exponents of $C$. With this convention one has $d_1={\rm mdr}(f)$.
\end{definition}

\begin{definition}
A $2$-syzygy reduced plane curve $C:f=0$ of degree $d$ with exponents $(d_1,d_2)$ is called \textbf{free} if $d_1+d_2=d-1$.
\end{definition}

\begin{remark} \label{rem:epsilon}
If $C$ is not free, then the shifts in the left-hand term of the above resolution can be written in the form
\[
e_j=d+d_{j+2}-1+\varepsilon_j,\qquad j=1,\ldots,s-2,
\]
for some integers $\varepsilon_j\geq1$; see \cite{Cremona}. In particular,
\[
d_1+d_2=d-1+\sum_{j=1}^{s-2}\varepsilon_j.
\]
Thus the integers $\varepsilon_j$ measure the deviation from freeness.
\end{remark}

\begin{definition}
Let $C:f=0$ be a $3$-syzygy reduced plane curve of degree $d$ with exponents $(d_1,d_2,d_3)$. We say that $C$ is \textbf{plus-one generated} (POG for short) if
\[
d_1+d_2=d.
\]
\end{definition}

\begin{remark}
A plus-one generated curve is called \textbf{minimal plus-one generated} (MPOG for short) if $d_3=d_2+1$, and it is called \textbf{nearly free} if $d_3=d_2$.
\end{remark}

For a reduced curve $C$, let
\begin{equation}
    \tau(C)=\sum_{p\in{\rm Sing}(C)}\tau_p,
    \label{eq:Trurina}
\end{equation}
where $\tau_p$ is the local Tjurina number at $p$. We will repeatedly use the following characterization of plus-one generated curves by Dimca and Sticlaru \cite{DS20}.

\begin{proposition}[Dimca--Sticlaru]
\label{Prop:Tau_POG}
Let $C:f=0$ be a reduced $3$-syzygy curve of degree $d\geq3$ with exponents $(d_1,d_2,d_3)$. Then $C$ is plus-one generated if and only if
\[
\tau(C)=(d-1)^2-d_1(d-d_1-1)-(d_3-d_2+1).
\]
\end{proposition}

\begin{definition}
For a plus-one generated curve with exponents $(d_1,d_2,d_3)$, the integer
\[
\nu(C):=d_3-d_2+1
\]
is called the \textbf{defect} of $C$.
\end{definition}

Thus a POG curve with $d_3>d_2$ has $\nu(C)\geq2$. The case $\nu(C)=2$ for conic arrangements with nodes, tacnodes, and ordinary triple points was classified in \cite[Theorem 3.1]{MP24}. The next natural case is $\nu(C)=3$, which is the subject of Section~\ref{sec:POG3}.

We also need the more general hierarchy of types introduced in \cite{NewH}.

\begin{definition}
Let $C:f=0$ be an $s$-syzygy reduced plane curve of degree $d$ with exponents $(d_1,d_2,\ldots,d_s)$. The \textbf{type} of $C$ is the non-negative integer
\[
t(C)=d_1+d_2+1-d.
\]
\end{definition}

A free curve has type $0$, while a plus-one generated curve has type $1$. For later comparison with non-POG arrangements we recall the first two higher types.

\begin{proposition} \label{prop:typ2}
Let $C\subset \mathbb P^2_{\mathbb C}$ be an $s$-syzygy reduced plane curve of degree $d$. In the notation of Remark~\ref{rem:epsilon}, curves of type $2$ occur in the following two forms:
\begin{itemize}
\item type $\mathbf{2A}$: $s=3$ and $\varepsilon_1=2$;
\item type $\mathbf{2B}$: $s=4$ and $\varepsilon_1=\varepsilon_2=1$.
\end{itemize}
For both forms one has $d_1+d_2=d+1$.
\end{proposition}

\begin{proposition} \label{prop:typ3}
Let $C\subset \mathbb P^2_{\mathbb C}$ be an $s$-syzygy reduced plane curve of degree $d$. In the notation of Remark~\ref{rem:epsilon}, curves of type $3$ occur in the following forms:
\begin{itemize}
\item type $\mathbf{3A}$: $s=3$ and $\varepsilon_1=3$;
\item type $\mathbf{3B}$: $s=4$, $\varepsilon_1=1$, and $\varepsilon_2=2$;
\item type $\mathbf{3B'}$: $s=4$, $\varepsilon_1=2$, and $\varepsilon_2=1$;
\item type $\mathbf{3C}$: $s=5$ and $\varepsilon_1=\varepsilon_2=\varepsilon_3=1$.
\end{itemize}
For all these forms one has $d_1+d_2=d+2$; see \cite{typ3}.
\end{proposition}

We now fix the two levels of combinatorial information used in the definition of Ziegler pairs.

\begin{definition}\label{def:WeakComb}
Let $C=\{C_1,\ldots,C_n\}\subset\mathbb P^2_{\mathbb C}$ be a reduced curve whose irreducible components are smooth. The \textbf{weak combinatorics} of $C$ is the vector
\[
\mathcal K(C)=(k_d,k_{d-1},\ldots,k_1;m_1,\ldots,m_p),
\]
where $k_i$ is the number of irreducible components of degree $i$, and $m_j$ is the number of singular points of topological type $M_j$.
\end{definition}

\begin{definition} \label{def:StrongComb}
Let $C\subset\mathbb P^2_{\mathbb C}$ be a reduced curve. The \textbf{strong combinatorial type} of $C$ is the datum
\[
\mathcal W(C)=(\mathbf i,\bar d,{\rm Sing}(C),\Sigma,\delta,\iota,\mathbf r),
\]
where
\begin{itemize}
\item the elements of $\mathbf i$ are in bijection with the irreducible components of $C$;
\item $\bar d:\mathbf i\rightarrow\mathbb N$ assigns to each component its degree;
\item ${\rm Sing}(C)$ is the set of singular points of $C$;
\item $\Sigma$ is the set of topological types occurring among points of ${\rm Sing}(C)$;
\item $\delta:{\rm Sing}(C)\rightarrow\Sigma$ assigns to each singular point its topological type;
\item $\iota(p)\subseteq\mathbf i$ is the set of components passing through $p$;
\item $\mathbf r=(r_j)_j$, where $r_j$ is the number of irreducible components containing exactly $j$ singular points.
\end{itemize}
\end{definition}

\begin{definition}
\label{iso}
Two reduced plane curves $C_1$ and $C_2$ have the same strong combinatorial type if there are bijections
\[
\phi_r:\mathbf i_1\rightarrow\mathbf i_2,
\qquad
\phi_S:{\rm Sing}(C_1)\rightarrow{\rm Sing}(C_2)
\]
such that $\bar d_1=\bar d_2\circ\phi_r$, $\delta_1=\delta_2\circ\phi_S$, and
\[
\iota_1(p)=\iota_2(\phi_S(p))
\quad\text{for every }p\in{\rm Sing}(C_1).
\]
In addition, the vectors $\mathbf r_1$ and $\mathbf r_2$ are required to coincide.
\end{definition}

\begin{remark}
The datum $\mathcal W(C)$ determines $\mathcal K(C)$, but the converse is not true. The weak combinatorics records only the numbers of components and singularities of each type, while $\mathcal W(C)$ also remembers the incidence relation between singular points and components.
\end{remark}

Finally, we recall the notion of the Poincar\'e-type polynomial that will be used in what follows. Let us emphasize that this notion was originally introduced in \cite{PokoraPoly} in the setting of free plane curves, including arrangements of conics and lines. It was later extended in \cite{BJP} to the broader class of plus-one generated arrangements of conics and lines, and finally in \cite{BJ} to the arbitrary reduced plane curve. For an $s$-syzygy conic arrangement $C$ of degree $d=2k\geq4$ and exponents $(d_1,d_2,d_3,\ldots,d_s)$ with $s\geq3$, define
\begin{equation}
\mathcal{P}(C,d_3;t):=1+2kt+\left(\sum_{r\geq2}(r-1)n_r+t_3+t_5+t_7+2j+2k-d_3\right)t^2,
\label{eq:CombPoly}
\end{equation}
where $n_r$ denotes the number of ordinary quasi-homogeneous $r$-tuple points, $t_i$ is the number of singularities of type $A_i$ for $i\in\{3,5,7\}$, and $j$ is the number of singularities of type $J_{2,0}$. We use Arnold's notation for local normal forms; see \cite{Arnold}.

\begin{corollary} \label{prop:CombPoly}
Let $C\subset\mathbb P^2_{\mathbb C}$ be a $3$-syzygy arrangement of $k\geq2$ conics with only $n_2$ nodes, $t_3$ tacnodes, and $n_3$ ordinary triple points, and let $(d_1,d_2,d_3)$ be its exponents. Then
\[
\mathcal P(C,d_3;t)=1+2kt+(n_2+t_3+2n_3+2k-d_3)t^2.
\]
\end{corollary}

\begin{proof}
In the singularity types considered here, nodes are ordinary double points, tacnodes are singularities of type $A_3$, and ordinary triple points contribute with coefficient $2$ in the sum $\sum_{r\geq2}(r-1)n_r$. Substituting these contributions into \eqref{eq:CombPoly} gives the stated formula.
\end{proof}

\section{Plus-one generated conic arrangements with only nodes, tacnodes, and ordinary triple intersection points with \texorpdfstring{$\nu(C)=3$}{nu(C)=3}} \label{sec:POG3}

We now study plus-one generated arrangements of $k\geq2$ smooth conics with defect $\nu(C)=3$. The strategy is to translate the homological assumption into numerical restrictions on $n_2,t_3,n_3$, and then compare these restrictions with geometric constraints coming from intersection theory. 
\noindent
We will use the following theorem of Dimca--Sernesi \cite[Theorem 2.1]{Sernesi}: if $C$ is a reduced plane curve of degree $d$ with only weighted homogeneous singularities, then 
\begin{equation}\tag{$\triangle$}\label{eqtriangle}
    \text{mdr}(C)\geq\alpha_C\cdot d-2,
\end{equation}
where $\alpha_C$ is the minimum of the Arnold exponents $\alpha_p$ of the singular points $p$ of $C$. 

Recall that a polynomial $f(x,y)=\sum_{u,v\in\mathbb{N}}c_{u,v}x^uy^v$ with local analytic coordinates $x,y$ centered at $p=(0,0)$ and $c_{u,v}\in\mathbb{C}$, is called \textit{weighted homogeneous} of type $(w_1,w_2;1)$ with weights $0<w_j\leq1/2$, if for every non-zero coefficient $c_{u,v}$, the following condition holds:
$$uw_1+vw_2=1.$$

\begin{definition}
For such a polynomial, the \textbf{Arnold exponent} (also known as the \textit{log-canonical threshold} $c_0(f)$) at the origin $p$ is the sum of its weights, i.e.
$$\alpha_p=w_1+w_2.$$
\end{definition}

\begin{example}
    Let $f(x,y)=x^3+xy^2$, which defines an ordinary triple point at $p=(0,0)$. Since $3w_1=1$ and $w_1+2w_2=1$ yield $w_1=w_2=1/3$, the Arnold exponent at $p$ is: $$\alpha_p=\frac{1}{3}+\frac{1}{3}=\frac{2}{3}.$$
\end{example}

For nodes, tacnodes, and ordinary triple points the corresponding Arnold exponents are $1$, $3/4$, and $2/3$, respectively. Hence the smallest possible Arnold exponent in our setting is $2/3$.
        \begin{proposition} \label{prop:d1}
            Let $C: f=0$ be a plus-one generated arrangement of $k\geq 2$ conics with only nodes, tacnodes, and ordinary triple intersection points as singularities. Then
            $$\frac{4}{3}k-2 \leq d_1\leq k.$$
            In particular, we have $k\leq 6$.
        \end{proposition}

\begin{proof}
For nodes, tacnodes, and ordinary triple points the Arnold exponents are respectively
\[
1,\qquad \frac34,\qquad \frac23.
\]
Hence, for any arrangement whose singularities are among these three types, one has
\[
\alpha_C\geq \frac23.
\]
By the Dimca--Sernesi bound \eqref{eqtriangle},
\[
d_1=\operatorname{mdr}(C)\geq \alpha_C\cdot 2k-2
\geq \frac23\cdot 2k-2
=
\frac43k-2.
\]
Moreover, since \(C\) is plus-one generated with exponents
\[
d_1\leq d_2\leq d_3,
\]
we have
\[
2d_1\leq d_1+d_2=2k.
\]
Therefore
\[
d_1\leq k.
\]
Thus
\[
\frac43k-2\leq d_1\leq k.
\]
In particular,
\[
k\leq 6.
\]
This proves the assertion.
\end{proof}


We now illustrate these preliminaries in the following example.
\begin{example} \label{ex:4k}
    Assume that we have an arrangement $C\subset\mathbb{P}^2_\mathbb{C}$ of $4$ conics with $n_2$ nodes, $t_3$ tacnodes and $n_3$ ordinary triple points as singularities which is a plus-one generated curve with defect $\nu(C)=3$.
Using the bounds for $\text{mdr}(C)=d_1$ established in Proposition~\ref{prop:d1}, for $k=4$ we obtain $d_1=4$. 
    Therefore, under our assumptions, we have $(d_1,d_2,d_3)=(4,4,6)$, and thus $\tau(C)=34$. 
    Using the formula for the total Tjurina number \eqref{eq:Trurina} and B\'ezout's theorem, the following system of equations holds:
    $$\begin{cases}
n_2+3t_3+4n_3=34 \\
n_2+2t_3+3n_3=4\cdot\binom{4}{2}=24
\end{cases}.$$
 Then $n_2=4-n_3$ and $t_3=10-n_3$. Since $n_3$ is a non-negative integer, the following weak combinatorics are numerically admissible:
    $$\mathcal{K}(C)=(n_2,t_3,n_3)\in\{(4,10,0),(3,9,1),(2,8,2),(1,7,3),(0,6,4)\}.$$
    Moreover, using Corollary~\ref{prop:CombPoly}, the combinatorial polynomial associated with $C$ is the same in each of the above cases, namely
$$\mathcal{P}(C,d_3;t)=1+8t+16t^2=(1+4t)^2.$$
The polynomial $\mathcal{P}$ splits over the rationals. Thus, according to \cite[Theorem 3.2]{BJP}, if an arrangement $C$ of $4$ conics can be realized as a plus-one generated curve with defect $\nu(C)=3$, then the exponents of $C$ must be $(4,4,6)$. 
\end{example}

\begin{remark} \label{rem:5k}
Arguing as in Example~3.4, assume that \(k=5\) and that \(C\) is plus-one generated with defect
\[
\nu(C)=3.
\]
Then Proposition \ref{prop:d1} gives
\[
d_1=5.
\]
Since \(d=10\) and \(d_1+d_2=d\), we get
\[
(d_1,d_2,d_3)=(5,5,7).
\]
Therefore Proposition \ref{Prop:Tau_POG} gives
\[
\tau(C)
=
(d-1)^2-d_1(d-d_1-1)-3
=
9^2-5\cdot 4-3
=
58.
\]
Thus
\[
n_2+3t_3+4n_3=58.
\]
On the other hand, B\'ezout's theorem gives
\[
n_2+2t_3+3n_3
=
2k(k-1)
=
40.
\]
Subtracting the two equations gives
\[
t_3+n_3=18.
\]
Combining this with B\'ezout's equation gives
\[
n_2+n_3=4.
\]
Hence, using only these two numerical equations, one obtains the preliminary list
\[
(n_2,t_3,n_3)\in
\{(4,18,0),(3,17,1),(2,16,2),(1,15,3),(0,14,4)\}.
\]
If the Hirzebruch-type inequality
\[
32k+4n_2+3n_3\geq 10t_3
\]
is also imposed, then the first possibility \((4,18,0)\) is excluded, since
\[
32\cdot 5+4\cdot 4+3\cdot 0
=
176
<
180
=
10\cdot 18.
\]
Consequently, after imposing the Hirzebruch-type inequality, the remaining numerical possibilities are
\[
(n_2,t_3,n_3)\in
\{(3,17,1),(2,16,2),(1,15,3),(0,14,4)\}.
\]
\end{remark}

We record an observation that will be useful for the classification.


\begin{lemma} \label{n2n3}
    Let $C\subset \mathbb{P}^2_{\mathbb{C}}$ be an arrangement of $k\geq 2$ smooth conics admitting only $n_2$ nodes, $t_3$ tacnodes and $n_3$ ordinary triple points. If $C$ is plus-one generated with defect $\nu (C)=3$, then $n_2+n_3\leq4.$
\end{lemma}

\begin{proof}
    By Proposition~\ref{Prop:Tau_POG}, if $C$ is POG of degree $d=2k\geq4$ and $\nu(C)=3$, then we obtain
    $$r^2-r(2k-1)+(2k-1)^2=\tau(C)+\nu(C)=n_2+3t_3+4n_3+3,$$
    where $r=\text{mdr}(C)=d_1.$ Recall that by B\'ezout's Theorem we have the following count:
    \begin{equation}
        4\cdot \binom{k}{2}=2k(k-1)=n_2+2t_3+3n_3.
        \label{Bezout}
    \end{equation}
    The binomial coefficient on the left-hand side of the above equation $\binom{k}{2}$ counts the number of pairs of curves from the set of $k\geq2$ conics. Combining the previous two equations, we obtain
    $$r^2-r(2k-1)+2k^2-2k-2-(t_3+n_3)=0,$$
    as a quadratic equation in the variable $r$. Since $C$ is plus-one generated, then the above equation must have a solution. Hence, its discriminant $\Delta_r$ is non-negative, i.e.
    $$(2k-1)^2-4\left(2k^2-2k-2-(t_3+n_3)\right)\geq0,$$
    which gives us
    $$t_3+n_3\geq k^2-k-\frac{9}{4}.$$
    Notice that
    $$4\cdot \binom{k}{2}=n_2+n_3+2(t_3+n_3)\geq n_2+n_3+2k^2-2k-\frac{9}{2},$$
    which leads to the desired inequality
    $$n_2+n_3\leq \left\lfloor\frac{9}{2}\right\rfloor=4.$$
\end{proof}

We next provide a combinatorial property of arrangements of at least three smooth conics having nodes, tacnodes, and ordinary triple points as singularities.

\begin{proposition} \label{prop:oneD4}
Let \(C=\{C_1,\ldots,C_k\}\) be an arrangement of \(k\geq 3\) smooth
conics in \(\mathbb P^2_{\mathbb C}\). Assume that the conics have one
common ordinary \(k\)-tuple point \(P\). Then the number of tacnodes of
the arrangement away from \(P\) is at most
\[
\binom{k}{2}=\frac{k(k-1)}{2}.
\]
In particular, for \(k=3\), an arrangement of three smooth conics with
one ordinary triple point has at most three tacnodes.
\end{proposition}

\begin{proof}
    Let $P$ be the only ordinary $k$-tuple point of the arrangement $C=\{C_1,\ldots,C_k\}\subset\mathbb{P}^2_\mathbb{C}$. Then, in particular, $C_i$ and $C_j$ meet transversely at $P$. Hence, by B\'ezout's theorem, the sum of the intersection multiplicities away from $P$ is
    $$\sum_{Q\in(C_i\cap C_j)\backslash\{P\}}I_Q(C_i,C_j) = \deg(C_i)\deg(C_j)-I_P(C_i,C_j)=3.$$
    Therefore, if $t_{ij}$ is the number of tacnodes in $C_i\cap C_j$ away from $P$, then
    $2t_{ij}\leq 3$. Hence $t_{ij}\leq 1$. Summing over all $\binom{k}{2}$ pairs $(i,j)$ gives
    $$\sum_{1\leq i<j\leq k}t_{ij}\leq \binom{k}{2}=\frac{k(k-1)}{2},$$
    which completes the proof.
\end{proof}

Before the main classification theorem, we obtain a non-trivial observation for arrangements of four smooth conics.

\begin{proposition}
\label{391}
The weak combinatorics
\[
        (n_2,t_3,n_3)=(3,9,1)
\]
is geometrically realizable over \(\mathbb C\) by an arrangement of four smooth conics.
\end{proposition}

\begin{proof}
Let \(\imath^2=-1\). We consider the following four conics in
\(\mathbb P^2_{\mathbb C}\):
\[
\begin{aligned}
Q_0&: q_0=y^2-4xz=0,\\
Q_1&: q_1=x^2-12\imath xy-146xz+12\imath yz+z^2=0,\\
Q_2&: q_2=256x^2-48\imath xy-137xz+12\imath yz+16z^2=0,\\
Q_3&: q_3=729x^2-108\imath xy-178xz+12\imath yz+9z^2=0.
\end{aligned}
\]
We denote by
\[
        \mathcal C=Q_0\cup Q_1\cup Q_2\cup Q_3
\]
the corresponding conic arrangement.

\noindent
We now describe the singular locus of \(\mathcal C\). The conic \(Q_0\)
has the parametrization
\[
        \phi(t)=(t^2:2t:1).
\]
Substituting this parametrization into the remaining three equations gives
\[
\begin{aligned}
q_1(\phi(t))&=(t^2-12\imath t-1)^2,\\
q_2(\phi(t))&=(16t^2-3\imath t-4)^2,\\
q_3(\phi(t))&=(27t^2-4\imath t-3)^2.
\end{aligned}
\]
The three quadratic factors are square-free and pairwise coprime. Therefore
\(Q_0\) is tangent to each of \(Q_1,Q_2,Q_3\) at two distinct points, and these
six tangency points are pairwise distinct. Consequently, the intersections
\[
        Q_0\cap Q_1,\qquad Q_0\cap Q_2,\qquad Q_0\cap Q_3
\]
contribute six tacnodes.

Next, the conics \(Q_1,Q_2,Q_3\) pass through the point
\[
        P=(0:1:0),
\]
whereas \(P\notin Q_0\), since \(q_0(P)=1\). In the affine chart \(y=1\),
with local coordinates \(u=x/y\) and \(v=z/y\), the linear parts of
\(q_1,q_2,q_3\) at \(P\) are proportional to
\[
        v-u,\qquad v-4u,\qquad v-9u,
\]
respectively. These three tangent directions are distinct, hence \(P\) is an
ordinary triple point of the arrangement.

The remaining pairwise intersections among \(Q_1,Q_2,Q_3\) are the following:
\[
\renewcommand{\arraystretch}{1.4}
\begin{array}{c|c|c}
\text{pair} & \text{node} & \text{tacnode}\\
\hline
(Q_1,Q_2)
&
\left(1:-\dfrac{929\imath}{60}:\dfrac{21}{5}\right)
&
\left(1:-\dfrac{33\imath}{4}:-2\right)
\\
(Q_1,Q_3)
&
\left(1:-\dfrac{151\imath}{12}:10\right)
&
\left(1:-\dfrac{28\imath}{3}:-3\right)
\\
(Q_2,Q_3)
&
\left(1:-\dfrac{3719\imath}{420}:-\dfrac{17}{7}\right)
&
\left(1:\dfrac{5\imath}{12}:6\right).
\end{array}
\]
A direct substitution shows that each of the above points lies on the indicated
pair of conics and on no other component of \(\mathcal C\). Moreover, the local
intersection multiplicities are
\[
\begin{aligned}
I_P(Q_1,Q_2)&=I_P(Q_1,Q_3)=I_P(Q_2,Q_3)=1,\\
I_{N_{12}}(Q_1,Q_2)&=I_{N_{13}}(Q_1,Q_3)=I_{N_{23}}(Q_2,Q_3)=1,\\
I_{T_{12}}(Q_1,Q_2)&=I_{T_{13}}(Q_1,Q_3)=I_{T_{23}}(Q_2,Q_3)=2.
\end{aligned}
\]
Thus \(N_{12},N_{13},N_{23}\) are nodes, while
\(T_{12},T_{13},T_{23}\) are tacnodes.

For each pair of smooth conics the total intersection multiplicity is equal to
\(4\), by B\'ezout's theorem. Hence the points listed above exhaust all pairwise
intersections among \(Q_1,Q_2,Q_3\). Combining this with the six tacnodes
coming from the intersections of \(Q_0\) with \(Q_1,Q_2,Q_3\), we obtain
\[
        n_2=3,\qquad t_3=6+3=9,\qquad n_3=1.
\]
Therefore the weak combinatorics
\[
        (n_2,t_3,n_3)=(3,9,1)
\]
is geometrically realizable over \(\mathbb C\).
\end{proof}

\begin{thm}
Let $C\subset\mathbb P^2_{\mathbb C}$ be an arrangement of $k$ smooth conics with only nodes, tacnodes, and ordinary triple points. Assume that $C$ is plus-one generated with defect $\nu(C)=3$ and that $k\in\{2,3,6\}$. Among the arrangements covered by the classification used in this paper, the only weak combinatorics occurring for such a curve is
\[
\mathcal K(C)=(k;n_2,t_3,n_3)=(3;4,4,0).
\]
Moreover, this weak combinatorics is realized by the arrangement
\[
f=(x^2+y^2-z^2)(4x^2+y^2-4z^2)(x^2+4y^2-4z^2),
\]
whose exponents are $(3,3,5)$.
\end{thm}

\begin{proof}
We first record the numerical restrictions used in the finite check. Lemma~\ref{n2n3} gives $n_2+n_3\leq4$. Combining B\'ezout's relation
\[
n_2+2t_3+3n_3=2k(k-1)
\]
with this inequality gives
\[
2k(k-1)=n_2+n_3+2(t_3+n_3)\leq4+2(t_3+n_3),
\]
and hence
\[
t_3+n_3\geq k(k-1)-2.
\]
Equivalently, the tacnode count is forced to be large when $k$ is large. A sharper upper bound follows from the same argument as in \cite[Remark 3.3]{Qarr}, with the ordinary triple points included in the count:
\[
t_3\leq \frac49k^2+\frac43k.
\]
For $k\geq4$ we also use the Hirzebruch-type inequality from \cite[Theorem 3.4]{BJP}, namely
\[
32k+4n_2+3n_3\geq10t_3.
\]
Together with B\'ezout's relation this gives the useful lower bound
\[
n_2\geq \left\lceil 4t_3-10k-\frac23k^2\right\rceil.
\]

For $k=2$, a direct inspection of the possible intersections of two smooth conics shows that no arrangement has defect $3$. The only plus-one generated case in this degree has weak combinatorics $(n_2,t_3,n_3)=(2,1,0)$, and it is minimal plus-one generated, so its defect is $2$ rather than $3$.

For \(k=6\), Proposition~3.3 gives
\[
d_1=6.
\]
Since \(C\) is plus-one generated and \(d=12\), we have
\[
d_1+d_2=12,
\]
and therefore
\[
(d_1,d_2,d_3)=(6,6,8).
\]
By Proposition \ref{Prop:Tau_POG},
\[
\tau(C)
=
(12-1)^2-6(12-6-1)-3
=
121-30-3
=
88.
\]
Hence
\[
n_2+3t_3+4n_3=88.
\]
On the other hand, B\'ezout's theorem gives
\[
n_2+2t_3+3n_3
=
2k(k-1)
=
60.
\]
Subtracting the second equation from the first one yields
\[
t_3+n_3=28.
\]
Moreover,
\[
n_2+2t_3+3n_3
=
n_2+n_3+2(t_3+n_3)
=
60,
\]
and hence
\[
n_2+n_3+56=60.
\]
Thus
\[
n_2+n_3=4.
\]
The upper bound
\[
t_3\leq \frac49k^2+\frac43k
\]
gives, for \(k=6\),
\[
t_3\leq \frac49\cdot 36+\frac43\cdot 6
=
16+8
=
24.
\]
Since
\[
t_3+n_3=28,
\]
we get
\[
n_3\geq 4.
\]
Together with
\[
n_2+n_3=4,
\]
this forces
\[
(n_2,t_3,n_3)=(0,24,4).
\]
Now the Hirzebruch-type inequality gives
\[
32k+4n_2+3n_3\geq 10t_3.
\]
For \(k=6\) and \((n_2,t_3,n_3)=(0,24,4)\), this becomes
\[
32\cdot 6+4\cdot 0+3\cdot 4\geq 10\cdot 24,
\]
that is
\[
204\geq 240,
\]
which is impossible. Therefore no plus-one generated arrangement with defect
\[
\nu(C)=3
\]
occurs for \(k=6\).

It remains to discuss $k=3$. The finite classification recorded in the appendix, together with the computation of the minimal free resolutions of the Milnor algebras, leaves exactly one plus-one generated arrangement with defect $3$. It has weak combinatorics
\[
(n_2,t_3,n_3)=(4,4,0).
\]
An explicit representative is
\[
f=(x^2+y^2-z^2)(4x^2+y^2-4z^2)(x^2+4y^2-4z^2).
\]
A direct computation in \texttt{SINGULAR} gives the exponents $(d_1,d_2,d_3)=(3,3,5)$. Since $d=6$, we have $d_1+d_2=d$ and $d_3-d_2+1=3$, so this arrangement is plus-one generated with defect $3$. This proves the assertion.
\end{proof}

\begin{remark}
\label{rem:speculation-k4-k5}
The discussion above using the numerical
formulae for \(\tau(C)\), B\'ezout's theorem, and the Hirzebruch-type
inequality is not sufficient to decide the existence of plus-one generated
arrangements with defect \(\nu(C)=3\) and $k \in \{4,5\}$. For \(k=4\) the numerical
conditions leave the following five weak combinatorics:
\[
(n_2,t_3,n_3)\in
\{(0,6,4),(1,7,3),(2,8,2),(3,9,1),(4,10,0)\}.
\]
If such an arrangement were plus-one generated with \(\nu(C)=3\), then,
by Example~\ref{ex:4k}, its exponents would necessarily be
\[
(d_1,d_2,d_3)=(4,4,6),
\]
and its total Tjurina number would be
\[
\tau(C)=34.
\]
Indeed, each of the five weak combinatorics listed above satisfies
\[
n_2+3t_3+4n_3=34.
\]
Thus, from the point of view of the total Tjurina number, none of these
possibilities can be excluded.

However, the known geometric constructions indicate a different phenomenon.
Let us recall that the most symmetric
realizations of the weak combinatorics
\[
(n_2,t_3,n_3)=(4,10,0)
\]
from \cite[Proposition 4.3]{BJP} are known to be of type \(2B\), and again they do not provide examples of
plus-one generated curves with exponents \((4,4,6)\).

The remaining weak combinatorics
\[
(n_2,t_3,n_3)=(3,9,1)
\]
is especially interesting. Proposition~\ref{391} shows that this weak
combinatorics can be geometrically realized over \(\mathbb C\) by four smooth
conics. Thus it cannot be discarded on purely geometric grounds. Nevertheless,
the construction given in Proposition~\ref{391} should be regarded only as just one realization of the prescribed singularity data. It does not, by itself, imply that the corresponding weak combinatorics cannot be constructed as a plus-one generated arrangement. The homological
question remains open for this particular weak combinatorics: one has to decide
whether there exists a realization with
\[
\operatorname{mdr}(C)=4
\]
and with exactly three minimal Jacobian syzygies of degrees $(4, 4, 6)$.
Equivalently, one has to determine whether the stratum of the realization space
corresponding to
\[
(n_2,t_3,n_3)=(3,9,1)
\]
contains a point for which the Milnor algebra has the minimal free resolution
expected for a plus-one generated curve with defect \(3\).

For \(k=5\) the situation is even more delicate. The numerical restrictions
give
\[
(n_2,t_3,n_3)\in
\{(0,14,4),(1,15,3),(2,16,2),(3,17,1)\}.
\]
If such an arrangement were plus-one generated with defect \(\nu(C)=3\), then
Remark~\ref{rem:5k} gives
\[
(d_1,d_2,d_3)=(5,5,7)
\]
and
\[
\tau(C)=58.
\]
Again, all the weak combinatorics listed above satisfy
\[
n_2+3t_3+4n_3=58.
\]
Thus the Tjurina number alone does not distinguish the possible POG cases from
the non-POG cases.

These observations lead to the following speculative conclusion. The
existence of plus-one generated conic arrangements with defect \(\nu(C)=3\)
for \(k=4\) or \(k=5\) is not ruled out by the present numerical restrictions.
Nevertheless, all currently available symmetric constructions lead either to
curves of type \(2B\) or to nearly free curves, rather than to curves with
exponents
\[
(4,4,6)
\quad\text{or}\quad
(5,5,7).
\]
Therefore, if POG arrangements with defect \(\nu(C)=3\) exist for \(k=4\) or \(k=5\),
\textit{they are likely to occur in less symmetric parts of the corresponding
realization spaces}. In practical terms, one should look for them by imposing
the singularity conditions first, and then studying the rank conditions on the
Jacobian syzygy module.
\end{remark}

Based on Proposition \ref{391} and Remark \ref{rem:speculation-k4-k5} we would like to conclude our section by the following difficult problem.
\begin{problem}
Does there exist an arrangement
\[
C\subset \mathbb P^2_{\mathbb C}
\]
of four smooth conics with only nodes, tacnodes, and ordinary triple points such that \(C\) is plus-one generated with defect
\[
\nu(C)=3?
\]
In particular, for the geometrically realizable weak combinatorics
\[
(n_2,t_3,n_3)=(3,9,1),
\]
does the corresponding realization space contain a point with
\[
\operatorname{mdr}(C)=4
\]
and with minimal Jacobian syzygies of degrees
\[
(4,4,6)?\]
\end{problem}

\section{Conic-line arrangements with only nodes, tacnodes, and ordinary triple intersection points} \label{sec:CL}

We now turn to conic-line arrangements. The point of this section is that B\'ezout's theorem determines the total intersection multiplicity, but it does not determine how the singularities are distributed among the components. Thus the same weak combinatorics may correspond to different incidence patterns. For arrangements consisting of two smooth conics and one line, the presence of an ordinary triple point imposes a particularly strong restriction: once the line passes transversely through the common point of the two conics, it has no remaining intersection multiplicity with either conic that could support a tangency.

\begin{lemma} \label{21121}
Let $C=Q_1\cup Q_2\cup\ell$ be an arrangement of two smooth conics and one line. Assume that $C$ has an ordinary triple point and at least one tacnode. Then every tacnode of $C$ must be formed by $Q_1$ and $Q_2$. In particular, neither $\ell$ and $Q_1$ nor $\ell$ and $Q_2$ can be tangent. Consequently, the weak combinatorics $(k,\ell;n_2,t_3,n_3)=(2,1;1,2,1)$ is not geometrically realizable.
\end{lemma}
\begin{proof}
    Let us assume that the arrangement contains at least one ordinary triple point.
    We shall analyze the following five possible graph representations:
    \begin{table}[h!]
    \centering
\renewcommand{\arraystretch}{1.4}
\begin{tabular}{>{\centering\arraybackslash}m{3cm} 
                >{\centering\arraybackslash}m{3cm}
                >{\centering\arraybackslash}m{3cm}
                >{\centering\arraybackslash}m{3cm}
                >{\centering\arraybackslash}m{3cm}}
\begin{tikzpicture}
                \draw (0,0) -- (1,1.2);
                \draw [style=dashed] (1,0) ellipse (1cm and 1.2cm);
                \draw [fill] (0,0) circle (3pt);
                \node at (-0.1,-0.5) {$Q_1$};
                \draw [fill] (2,0) circle (3pt);
                \node at (2.2,-0.5) {$Q_2$};
                \draw [fill=white] (1,1.2) circle (3pt);
                \node at (1.4, 1.25) {$\ell$};
            \end{tikzpicture} 
            &
\begin{tikzpicture}
                \draw (0,0) -- (2,0);
                \draw [style=dashed] (1,0) ellipse (1cm and 1.2cm);
                \draw [fill] (0,0) circle (3pt);
                \node at (-0.1,-0.5) {$Q_1$};
                \draw [fill] (2,0) circle (3pt);
                \node at (2.2,-0.5) {$Q_2$};
                \draw [fill=white] (1,1.2) circle (3pt);
                \node at (1.4, 1.25) {$\ell$};
            \end{tikzpicture} &
\begin{tikzpicture}
                \draw (0,0) -- (1,1.2) -- (2,0);
                \draw [style=dashed] (1,0) ellipse (1cm and 1.2cm);
                \draw [fill] (0,0) circle (3pt);
                \node at (-0.1,-0.5) {$Q_1$};
                \draw [fill] (2,0) circle (3pt);
                \node at (2.2,-0.5) {$Q_2$};
                \draw [fill=white] (1,1.2) circle (3pt);
                \node at (1.4, 1.25) {$\ell$};
            \end{tikzpicture} &
            \begin{tikzpicture}
                \draw (2,0) -- (0,0) -- (1,1.2);
                \draw [style=dashed] (1,0) ellipse (1cm and 1.2cm);
                \draw [fill] (0,0) circle (3pt);
                \node at (-0.1,-0.5) {$Q_1$};
                \draw [fill] (2,0) circle (3pt);
                \node at (2.2,-0.5) {$Q_2$};
                \draw [fill=white] (1,1.2) circle (3pt);
                \node at (1.4, 1.25) {$\ell$};
            \end{tikzpicture} &
            \begin{tikzpicture}
                \draw [style={double,double distance=2pt}] (0,0) -- (2,0);
                \draw [style=dashed] (1,0) ellipse (1cm and 1.2cm);
                \draw [fill] (0,0) circle (3pt);
                \node at (-0.1,-0.5) {$Q_1$};
                \draw [fill] (2,0) circle (3pt);
                \node at (2.2,-0.5) {$Q_2$};
                \draw [fill=white] (1,1.2) circle (3pt);
                \node at (1.4, 1.25) {$\ell$};
            \end{tikzpicture}
\end{tabular}
\end{table}

\noindent
where 
\begin{itemize}
    \item an open point labeled $\ell$ denotes a line,
    \item a solid point labeled $Q_i$ represents a conic,
    \item a solid edge between vertices indicates a tacnode at the intersection of the corresponding two curves,
    \item a dashed connection indicates an ordinary triple at the intersection of the corresponding three curves,
    \item the absence of an edge between two vertices indicates that the curves intersect at nodes, according to B\'ezout's theorem.
\end{itemize}
Let $O$ be the ordinary triple point. Since the three local branches meet transversely at $O$, we have
\[
I_O(\ell,Q_i)=1\qquad (i=1,2).
\]
On the other hand, B\'ezout's theorem gives
\[
\sum_{P\in \ell\cap Q_i} I_P(\ell,Q_i)=\deg(\ell)\deg(Q_i)=2.
\]
Thus, after the contribution at $O$, only one unit of intersection multiplicity remains for the pair $(\ell,Q_i)$. A tangency between $\ell$ and $Q_i$ away from $O$ would contribute multiplicity $2$, and a tangency at $O$ would contradict the assumption that $O$ is ordinary. Hence the line cannot be tangent to either conic. This eliminates every graph in which the open vertex is joined to a filled vertex by a solid edge, namely the first, third, and fourth graphs.

The fifth graph has two tacnodes between $Q_1$ and $Q_2$. Each of them contributes intersection multiplicity $2$, so the two tacnodes already account for the full B\'ezout number
\[
\deg(Q_1)\deg(Q_2)=4.
\]
There is then no remaining intersection multiplicity for an additional transverse intersection of $Q_1$ and $Q_2$ at the ordinary triple point. Therefore the fifth graph is also impossible.

\noindent
\begin{minipage}[t]{0.6\textwidth}
\vspace{0pt}
\ \ \ \ The only remaining graph is the second one: the tacnode is formed by the two conics, while the line meets both conics transversely at the ordinary triple point. Substituting $t_3=n_3=1$ into B\'ezout's relation $8=n_2+2t_3+3n_3$ gives $n_2=3$. The following equations give a realization of this admissible case; the picture is shown in Figure~\ref{fig:C_curve}.
$$\ell: \ x+3y=0; \hspace{0.3cm} Q_1: \ x^2+y^2-yz-2z^2=0; \hspace{0.3cm} Q_2: \ x^2+3y^2-3z^2=0.$$
This completes the proof.
\end{minipage}
\hfill
\begin{minipage}[t]{0.4\textwidth}
\vspace{0pt}
\centering
\begin{tikzpicture}[scale=1]
    \draw [line width=1.2pt, rotate=0] (0,0.5) ellipse (1cm and 1.5cm); 
    \node at (1.3, 1.6) {$Q_1$};
                \draw [line width=1.2pt, rotate=0] (0,0) ellipse (2cm and 1cm);
                \node at (2.4,0.3) {$Q_2$};
                \draw [line width=1.2pt] (-1.7,1.5) -- (1.6,-1.2);
                \node at (-2,1.5) {$\ell$};
            \end{tikzpicture}
            
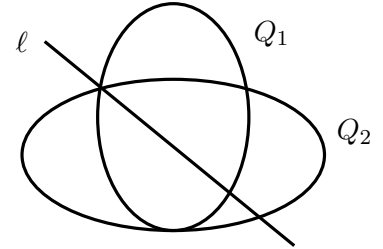
\captionof{figure}{$\mathcal{K}(C)=(2,1;3,1,1)$.}
\label{fig:C_curve}
\end{minipage}
\end{proof}

\begin{remark}
    If $t_3=1$ and $n_3=2$, then the B\'ezout relation gives $n_2=0$, hence $\mathcal{K}(C)=(2,1;0,1,2)$. This arrangement is obtained from the same two conics as in Figure~\ref{fig:C_curve}, but by choosing the line $\ell:2y-z=0$.
\end{remark}

\section{Weak and strong Ziegler pairs of arrangements of conics and lines} \label{Ziegler}

In this section we list the weak and strong Ziegler pairs found in the database of conic and conic-line arrangements of degree at most $6$ recorded in the \hyperref[sec:data]{Appendix}. We first fix the terminology.

\begin{definition}
Let $C_1,C_2\subset\mathbb P^2_{\mathbb C}$ be reduced curves. We say that $C_1$ and $C_2$ form a \textbf{weak Ziegler pair} if
\[
\mathcal K(C_1)=\mathcal K(C_2),
\]
but the Milnor algebras $M(C_1)$ and $M(C_2)$ have different minimal graded free resolutions. Equivalently, in the examples considered below, the Jacobian syzygy modules ${\rm AR}(C_1)$ and ${\rm AR}(C_2)$ are not isomorphic as graded $S$-modules.
\end{definition}

\begin{definition}
We say that \((C_1,C_2)\) is a \textbf{strong Ziegler pair} if \(C_1\) and
\(C_2\) have the same strong combinatorial type in the sense of
Definition \ref{iso}, but the corresponding Jacobian syzygy modules are not
isomorphic, or equivalently the Milnor algebras have different minimal
graded free resolutions.
\end{definition}

\begin{remark}
    We say that three reduced curves \(C_1,C_2, C_3\subset\mathbb P^2_{\mathbb{C}}\) form a {\bf weak Ziegler triple} if every pair \[(C_1,C_2), \qquad (C_1,C_3), \qquad (C_2,C_3)\] forms a weak Ziegler pair.
\end{remark}

We begin with the weak Ziegler pairs. The relevant weak combinatorics are:
\begin{enumerate}
    \item $\mathcal{K}(\mathcal{C}_1)=(2,1;4,2,0)$,
    \item $\mathcal{K}(\mathcal{C}_2)=(1,3;6,0,1)$,
    \item $\mathcal{K}(\mathcal{C}_3)=(3,0;6,3,0)$,
    \item $\mathcal{K}(\mathcal{C}_4)=(3,0;4,4,0)$,
    \item $\mathcal{K}(\mathcal{C}_5)=(1,4;11,0,1)$,
    \item $\mathcal{K}(\mathcal{C}_6)=(2,2;7,0,2)$,
    \item $\mathcal{K}(\mathcal{C}_7)=(2,2;9,2,0)$,
\end{enumerate}
where $\mathcal K(\mathcal C)=(k,\ell;n_2,t_3,n_3)$, $k$ denotes the number of conics and $\ell$ denotes the number of lines. The database in \hyperref[sec:data]{Appendix} gives the following coincidences of weak combinatorics.

\begin{proposition}\label{prop:weak-same-comb}
The following identities between weak combinatorics hold:
    \begin{enumerate}[label=\normalfont{\arabic*)}]
        \item $\mathcal{K}\left(W_{4,1}^{2,1}\right)=\mathcal{K}\left(W_{4,2}^{2,1}\right)=\mathcal{K}(\mathcal{C}_1)$,
        \item $\mathcal{K}\left(W_{3,1}^{1,3}\right)=\mathcal{K}\left(W_{3,2}^{1,3}\right)=\mathcal{K}(\mathcal{C}_2)$,
        \item $\mathcal{K}\left(W_{7,1}^{3,0}\right)=\mathcal{K}\left(W_{7,2}^{3,0}\right)=\mathcal{K}(\mathcal{C}_3)$,
        \item $\mathcal{K}\left(W_{14,1}^{3,0}\right)=\mathcal{K}\left(W_{14,2}^{3,0}\right)=\mathcal{K}(\mathcal{C}_4)$,
        \item $\mathcal{K}\left(W_{3,1}^{1,4}\right)=\mathcal{K}\left(W_{3,2}^{1,4}\right)=\mathcal{K}(\mathcal{C}_5)$,
        \item $\mathcal{K}\left(W_{4,1}^{2,2}\right)=\mathcal{K}\left(W_{4,2}^{2,2}\right)=\mathcal{K}\left(W_{4,3}^{2,2}\right)=\mathcal{K}(\mathcal{C}_6)$,
        \item $\mathcal{K}\left(W_{5,1}^{2,2}\right)=\mathcal{K}\left(W_{5,2}^{2,2}\right)=\mathcal{K}\left(W_{5,3}^{2,2}\right)=\mathcal{K}(\mathcal{C}_7)$.
    \end{enumerate}
\end{proposition}

\begin{proof}
The equalities are obtained by inspecting the singular points of the arrangements listed in the database. The figures below give real pictures of the corresponding configurations. Some nodes are not visible in the real affine chart: in Figure~\ref{fig:4} and Figure~\ref{fig:7} the line at infinity and the conic meet in two complex nodes; in the second arrangement in Figure~\ref{fig:5}, each pair of conics has two complex nodes and one real tacnode; and in the second arrangement in Figure~\ref{fig:W30-14}, only one pair of conics contributes two complex nodes. Taking these complex intersection points into account gives exactly the weak combinatorics stated above.
\end{proof}

\tikzset{
  curve/.style={black, very thick},
}
\tikzset{
  infinity/.style={black, very thick},
}

\begin{figure}[ht!]
\begin{minipage}[t]{0.5\textwidth}
\centering
\begin{tikzpicture}[scale=1]
\draw[curve] (-2,0.5) -- (2,0.5);
\draw[curve, domain=-1.4:1.4, samples=200]
  plot ({cosh(\x)}, {sinh(\x)});
\draw[curve, domain=-1.4:1.4, samples=200]
  plot ({-cosh(\x)}, {sinh(\x)});
\draw[curve] (0,0) circle (1);
\end{tikzpicture}
\end{minipage}
\begin{minipage}[t]{0.5\textwidth}
\centering
\begin{tikzpicture}[scale=1]
\draw[curve] (-3,-1) -- (1,3);
\draw[curve] (0,0) ellipse (1.732 and 1);
\draw[curve] (0,0) ellipse (1 and 1.732);
\end{tikzpicture}
\end{minipage}
\caption{Geometric realizations of arrangements of combinatorial type $W^{2,1}_{4,1}$ and $W^{2,1}_{4,2}$.}
\end{figure}
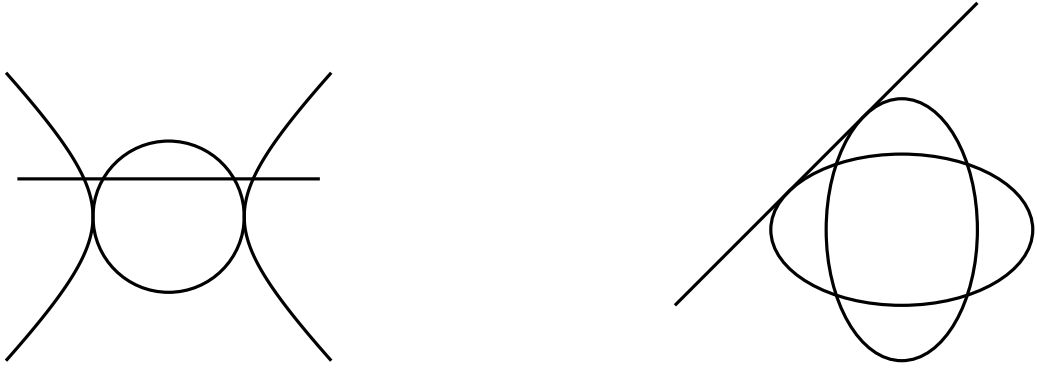

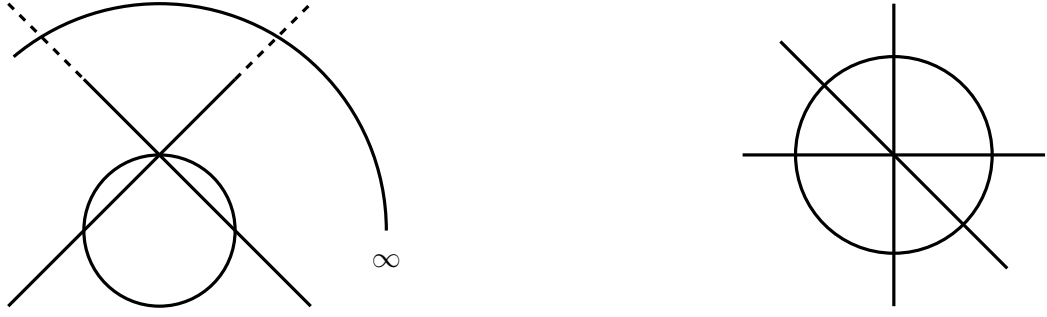
\begin{figure}[ht!]
\begin{minipage}[t]{0.5\textwidth}
\centering
\begin{tikzpicture}[scale=1]
\draw[curve] (-1,2) -- (2,-1);
\draw[curve, dashed] (-2,3) -- (-1,2);
\draw[curve] (-2,-1) -- (1,2);
\draw[curve, dashed] (1,2) -- (2,3);
\draw[curve] (0,0) circle (1);
  \draw[infinity] (3,0) arc[start angle=0, end angle=130, radius=3];
\node at (3,-0.4) {$\infty$};
\end{tikzpicture}
\end{minipage}
\begin{minipage}[t]{0.5\textwidth}
\centering
\begin{tikzpicture}[scale=1]
\draw[curve] (0,-2) -- (0,2);
\draw[curve] (-2,0) -- (2,0);
\draw[curve] (-1.5,1.5) -- (1.5,-1.5);
\draw[curve] (0,0) circle (1.3);
\end{tikzpicture}
\end{minipage}
\caption{Geometric realizations of arrangements of combinatorial type $W^{1,3}_{3,1}$ and $W^{1,3}_{3,2}$.}
\label{fig:4}
\end{figure}

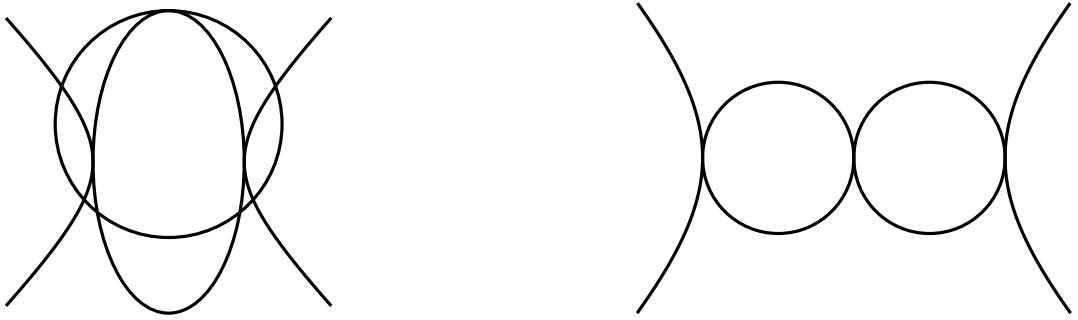
\begin{figure}[ht!]
\begin{minipage}[t]{0.5\textwidth}
\centering
\begin{tikzpicture}[scale=1]
\draw[curve] (0,0) ellipse (1 and 2);
\draw[curve, domain=-1.4:1.4, samples=200]
  plot ({cosh(\x)}, {sinh(\x)});
\draw[curve, domain=-1.4:1.4, samples=200]
  plot ({-cosh(\x)}, {sinh(\x)});
\draw[curve] (0,0.5) circle (1.5);
\end{tikzpicture}
\end{minipage}
\begin{minipage}[t]{0.5\textwidth}
\centering
\begin{tikzpicture}[scale=1]
\draw[curve] (1,0) circle (1);
\draw[curve] (-1,0) circle (1);
\draw[curve, domain=-0.9:0.9, samples=200]
  plot ({2*cosh(\x)}, {2*sinh(\x)});
\draw[curve, domain=-0.9:0.9, samples=200]
  plot ({-2*cosh(\x)}, {2*sinh(\x)});
\end{tikzpicture}
\end{minipage}
\caption{Geometric realizations of arrangements of combinatorial type $W^{3,0}_{7,1}$ and $W^{3,0}_{7,2}$.}
\label{fig:5}
\end{figure}

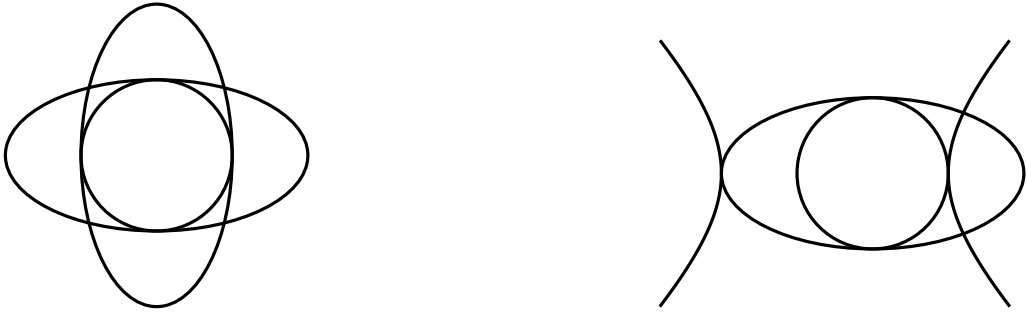
\begin{figure}[ht!]
\begin{minipage}[t]{0.5\textwidth}
\centering
\begin{tikzpicture}[scale=1]
\draw[curve] (0,0) ellipse (1 and 2);
\draw[curve] (0,0) circle (1);
\draw[curve] (0,0) ellipse (2 and 1);
\end{tikzpicture}
\end{minipage}
\begin{minipage}[t]{0.5\textwidth}
\centering
\begin{tikzpicture}[scale=1]
\draw[curve] (0,0) circle (1);
\draw[curve] (0,0) ellipse (2 and 1);
\draw[curve, domain=-1:1, samples=200]
  plot ({-0.5+1.5*cosh(\x)}, {1.5*sinh(\x)});
\draw[curve, domain=-1:1, samples=200]
  plot ({-0.5-1.5*cosh(\x)}, {1.5*sinh(\x)});
\end{tikzpicture}
\end{minipage}
\caption{Geometric realizations of arrangements of combinatorial type $W^{3,0}_{14,1}$ and $W^{3,0}_{14,2}$.}
\label{fig:W30-14}
\end{figure}

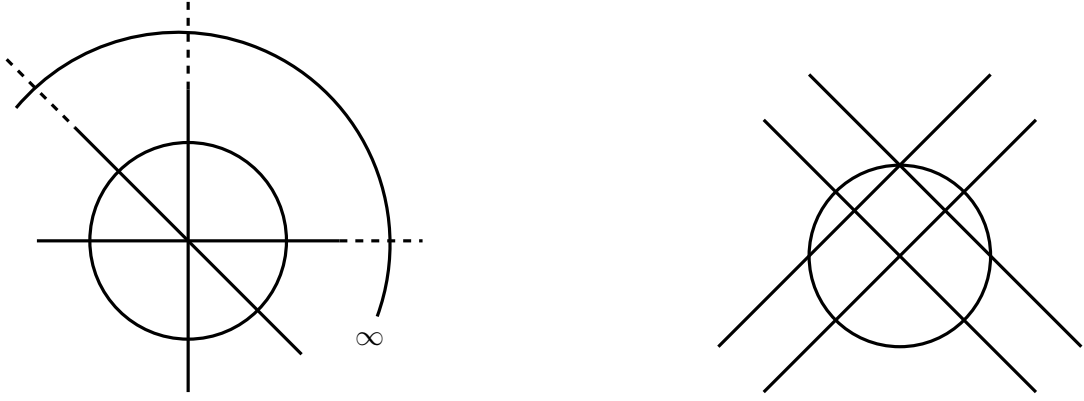
\begin{figure}[ht!]
\begin{minipage}[t]{0.5\textwidth}
\centering
\begin{tikzpicture}[scale=1]
\draw[curve] (0,-2) -- (0,2);
\draw[curve, dashed] (0,2) -- (0,3.2);
\draw[curve] (-2,0) -- (2,0);
\draw[curve, dashed] (2,0) -- (3.1,0);
\draw[curve] (-1.5,1.5) -- (1.5,-1.5);
\draw[curve, dashed] (-2.4,2.4) -- (-1.5,1.5);
\draw[curve] (0,0) circle (1.3);
\draw[infinity] (2.5,-1) arc[start angle=-20, end angle=140, radius=2.8];
\node at (2.4,-1.3) {$\infty$};
\end{tikzpicture}
\end{minipage}
\begin{minipage}[t]{0.5\textwidth}
\centering
\begin{tikzpicture}[scale=1.2]
\draw[curve] (-1.5,1.5) -- (1.5,-1.5);
\draw[curve] (-1.5,-1.5) -- (1.5,1.5);
\draw[curve] (-1,2) -- (2,-1);
\draw[curve] (-2,-1) -- (1,2);
\draw[curve] (0,0) circle (1);
\end{tikzpicture}
\end{minipage}
\caption{Geometric realizations of arrangements of combinatorial type $W^{1,4}_{3,1}$ and $W^{1,4}_{3,2}$.}
\label{fig:7}
\end{figure}

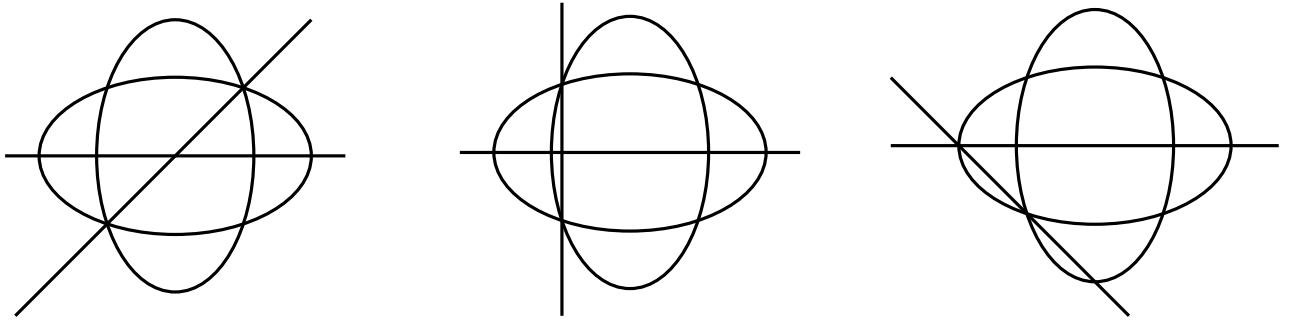
\begin{figure}[ht!]
\begin{minipage}[t]{0.33\textwidth}
\centering
\begin{tikzpicture}[scale=0.9]
\draw[curve] (-2.35,-2.35) -- (2,2);
\draw[curve] (-2.5,0) -- (2.5,0);
\draw[curve] (0,0) ellipse (2 and 1.155);
\draw[curve] (0,0) ellipse (1.155 and 2);
\end{tikzpicture}
\end{minipage}
\begin{minipage}[t]{0.33\textwidth}
\centering
\begin{tikzpicture}[scale=0.9]
\draw[curve] (-2.5,0) -- (2.5,0);
\draw[curve] (-1,-2.4) -- (-1,2.2);
\draw[curve] (0,0) ellipse (2 and 1.155);
\draw[curve] (0,0) ellipse (1.155 and 2);
\end{tikzpicture}
\end{minipage}
\begin{minipage}[t]{0.33\textwidth}
\centering
\begin{tikzpicture}[scale=0.9]
\draw[curve] (-3,0) -- (2.7,0);
\draw[curve] (-3,1) -- (0.5,-2.5);
\draw[curve] (0,0) ellipse (2 and 1.155);
\draw[curve] (0,0) ellipse (1.155 and 2);
\end{tikzpicture}
\end{minipage}
\caption{Geometric realizations of arrangements of combinatorial type $W^{2,2}_{4,1}$, $W^{2,2}_{4,2}$ and $W^{2,2}_{4,3}$.}
\label{fig:8}
\end{figure}

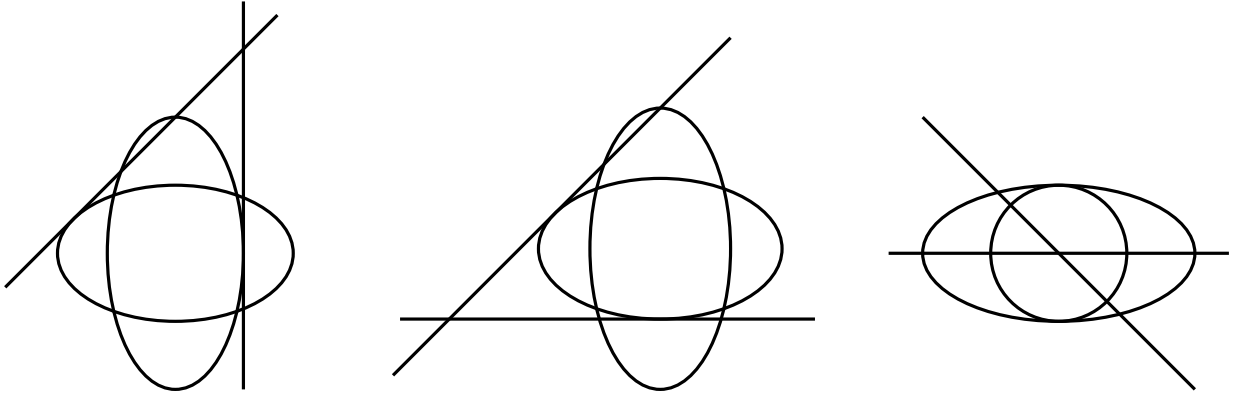
\begin{figure}[ht!]
\begin{minipage}[t]{0.33\textwidth}
\centering
\begin{tikzpicture}[scale=0.9]
\draw[curve] (-2.5,-0.5) -- (1.5,3.5);
\draw[curve] (1,-2) -- (1,3.7); 
\draw[curve] (0,0) ellipse (1 and 2);
\draw[curve] (0,0) ellipse (1.732 and 1);
\end{tikzpicture}
\end{minipage}
\begin{minipage}[t]{0.33\textwidth}
\centering
\begin{tikzpicture}[scale=0.93]
\draw[curve] (-3.8,-1.8) -- (1,3);
\draw[curve] (-3.7,-1) -- (2.2,-1);
\draw[curve] (0,0) ellipse (1 and 2);
\draw[curve] (0,0) ellipse (1.732 and 1);
\end{tikzpicture}
\end{minipage}
\begin{minipage}[t]{0.33\textwidth}
\centering
\begin{tikzpicture}[scale=0.9]
\draw[curve] (-2.5,0) -- (2.5,0);
\draw[curve] (-2,2) -- (2,-2);
\draw[curve] (0,0) circle (1);
\draw[curve] (0,0) ellipse (2 and 1);
\end{tikzpicture}
\end{minipage}
\caption{Geometric realizations of arrangements of combinatorial type $W^{2,2}_{5,1}$, $W^{2,2}_{5,2}$ and $W^{2,2}_{5,3}$.}
\label{fig:W22-5}
\end{figure}

\newpage
\begin{remark}
    It is noteworthy that in Figure~\ref{fig:8}, the first two curves not only have the same weak combinatorics, but they have the same strong combinatorial type, discussed later in this section.
\end{remark}
Thus the arrangements in each listed pair have the same weak combinatorics. In most cases they differ at the level of the stronger incidence vector $\mathbf r$ from Definition~\ref{def:StrongComb}. We now compare their homological data.

\begin{proposition}\label{prop:weak-ziegler-pairs}
The pairs (resp. triples) of arrangements listed in Proposition~\ref{prop:weak-same-comb} form weak Ziegler pairs (resp. triples).
\end{proposition}
\newpage
\begin{proof}
For each arrangement we compute the minimal graded free resolution of its Milnor algebra in \texttt{SINGULAR}. The results are as follows:

\begin{itemize}[leftmargin=3.5cm, labelsep=0.5cm]
    \item[$\left(W^{2,1}_{4,1}\right)$:] \quad $0\r S(-9)\r S(-8)\oplus S(-7)\oplus S(-6)\r S^3(-4)\r S$,
\item[$\left(W^{2,1}_{4,2}\right)$:] \quad $0\r S^2(-8)\r S^4(-7)\r S^3(-4)\r S$,
    \item[$\left(W^{1,3}_{3,1}\right)$:]
    \quad $0\r S^2(-8)\r S^4(-7)\r S^3(-4)\r S$,
\item[$\left(W^{1,3}_{3,2}\right)$:]
\quad $0\r S(-9)\r S(-8)\oplus S(-7)\oplus S(-6)\r S^3(-4)\r S$,
    \item[$\left(W^{3,0}_{7,1}\right)$:]
    \quad $0\r S(-11)\r S^2(-9)\oplus S(-8)\r S^3(-5) \r S$,
\item[$\left(W^{3,0}_{7,2}\right)$:]
\quad $0\r S^3(-10)\r S^5(-9)\r S^3(-5)\r S$, 
\item[$\left(W^{3,0}_{14,1}\right)$:]
\quad $0\r S(-11)\r S(-10)\oplus S^2(-8)\r S^3(-5)\r S$,
\item[$\left(W^{3,0}_{14,2}\right)$:]
\quad $0\r S^2(-10)\r S^3(-9)\oplus S(-8)\r S^3(-5)\r S$,
\item[$\left(W^{1,4}_{3,1}\right)$:]
\quad $0\r S(-11)\r S^2(-9)\oplus S(-8)\r S^3(-5)\r S$,
\item[$\left(W^{1,4}_{3,2}\right)$:]
\quad $0\r S^3(-10)\r S^5(-9)\r S^3(-5)\r S$, 
\item[$\left(W^{2,2}_{4,1}\right)$:]
\quad $0\r S(-11)\r S^2(-9)\oplus S(-8)\r S^3(-5)\r S$,
\item[$\left(W^{2,2}_{4,2}\right)$:]
\quad $0\r S(-11)\oplus S(-10)\r S(-10)\oplus S^2(-9)\oplus S(-8)\r S^3(-5)\r S$,
\item[$\left(W^{2,2}_{4,3}\right)$:]
\quad $0\r S^3(-10)\r S^5(-9)\r S^3(-5)\r S$, 
\item[$\left(W^{2,2}_{5,1}\right)$:]
\quad $0\r S^3(-10)\r S^5(-9)\r S^3(-5)\r S$,
\item[$\left(W^{2,2}_{5,2}\right)$:]
\quad $0\r S(-11)\r S^2(-9)\oplus S(-8)\r S^3(-5)\r S$,
\item[$\left(W^{2,2}_{5,3}\right)$:]
\quad $0\r S(-11)\oplus S(-10)\r S(-10)\oplus S^2(-9)\oplus S(-8)\r S^3(-5)\r S$.

\end{itemize}

\noindent
    Since each listed pair (resp. triple) has the same weak combinatorics but different minimal graded free resolutions of the corresponding Milnor algebras, each such pair forms a weak Ziegler pair (resp. triple).
\end{proof}

We now pass to strong Ziegler pairs. Here equality of weak combinatorics is not enough; we must also
check that the incidence data encoded by \(W\) agree. The relevant weak combinatorics are:
\[
\mathcal{K}(H)=(3,0;6,0,2),
\]
\[
\mathcal{K}(I)=(3,0;4,1,2),
\]
and
\[
\mathcal{K}(J)=(2,2;7,0,2).
\]
The arrangements
\[
W^{3,0}_{5,1},\quad W^{3,0}_{5,2}
\]
have weak combinatorics \(\mathcal{K}(H)\), the arrangements
\[
W^{3,0}_{8,1},\quad W^{3,0}_{8,2}
\]
have weak combinatorics \(\mathcal{K}(I)\), while the arrangements
\[
W^{2,2}_{4,1},\quad W^{2,2}_{4,2}
\]
have weak combinatorics \(\mathcal{K}(J)\).

\begin{proposition}\label{prop:strong-same-weak}
The pairs
\[
W^{3,0}_{5,1},\quad W^{3,0}_{5,2},
\]
\[
W^{3,0}_{8,1},\quad W^{3,0}_{8,2},
\]
and
\[
W^{2,2}_{4,1},\quad W^{2,2}_{4,2}
\]
have the same weak combinatorics pairwise.
\end{proposition}

\begin{proof}
For the first two pairs this follows from the direct count of nodes, tacnodes, and ordinary triple points in the displayed equations.

        \begin{alignat*}{2}
        \left(W^{3,0}_{5,1}\right)\!:\;&\quad (x^2+2y^2-3z^2)(x^2+y^2-2z^2)(x^2+yz-2z^2)=0, \\
\left(W^{3,0}_{5,2}\right)\!:\;&\quad (x^2+2y^2-3z^2)(x^2+y^2-2z^2)(x^2+2xz+2yz-z^2)=0, \\
\left(W^{3,0}_{8,1}\right)\!:\;&\quad (x^2-y^2-z^2)(x^2+y^2+xz-2z^2)(x^2+3y^2-6z^2)=0,\\
\left(W^{3,0}_{8,2}\right)\!:\;&\quad (x^2+3y^2-3z^2)(6x^2+6y^2-2\sqrt{3}xz+6\sqrt{3}yz-15z^2)(x^2-y^2-2z^2)=0.
\end{alignat*}

\noindent
For the third pair, the Appendix gives
\[
W^{2,2}_{4,1}=(2,2;7,0,2),
\qquad
W^{2,2}_{4,2}=(2,2;7,0,2).
\]
Thus
\[
\mathcal{K}\left(W^{2,2}_{4,1}\right)=K\left(W^{2,2}_{4,2}\right)=(2,2;7,0,2).
\]
This proves the assertion.
\end{proof}

\tikzset{
  curve/.style={black, very thick},
}
\tikzset{
  infinity/.style={black, very thick},
}

\begin{figure}[h!]
\begin{minipage}[t]{0.5\textwidth}
\centering
\begin{tikzpicture}[scale=1]
\draw[curve] (0,0) ellipse (1.732 and 1.224);
\draw[curve] (0,0) circle (1.414);
\draw[curve, domain=-2:2, smooth] plot (\x,{2-\x*\x});
\end{tikzpicture}
\end{minipage}
\begin{minipage}[t]{0.5\textwidth}
\centering
\begin{tikzpicture}[scale=1]
\draw[curve] (0,0) ellipse (1.732 and 1.224);
\draw[curve] (0,0) circle (1.414);
\draw[curve, domain=-2.9:1.4, smooth] plot (\x,{(1-\x*\x-2*\x)/2});
\end{tikzpicture}
\end{minipage}
\caption{Geometric realizations of arrangements of combinatorial type $W^{3,0}_{5,1}$ and $W^{3,0}_{5,2}$.}
\end{figure}
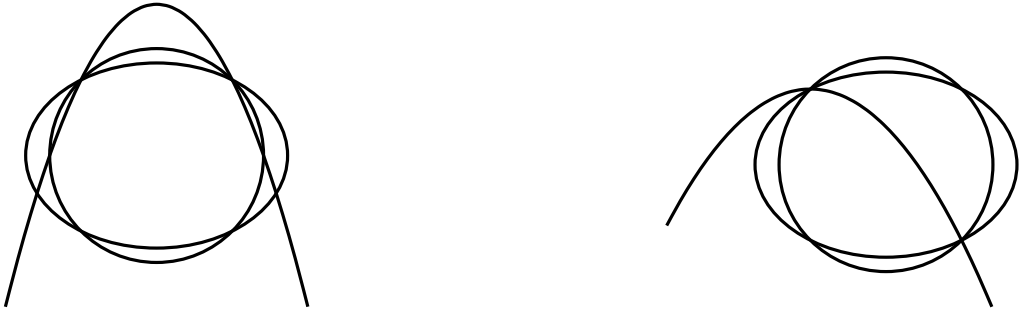

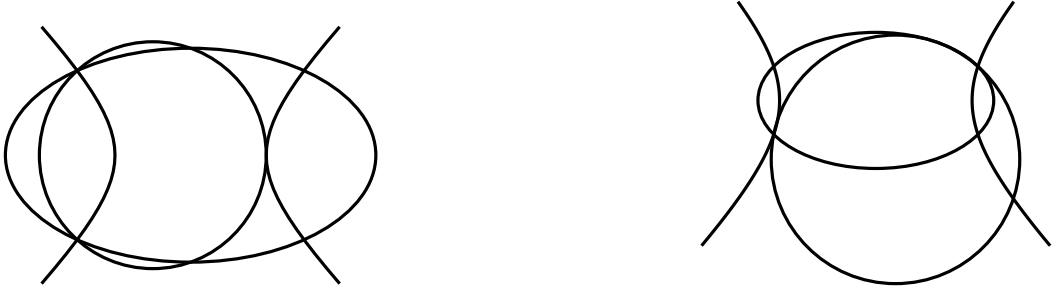
\begin{figure}[h!]
\begin{minipage}[t]{0.5\textwidth}
\centering
\begin{tikzpicture}[scale=1]
        \draw[curve, domain=-1.3:1.3] plot ({cosh(\x)}, {sinh(\x)});
        \draw[curve, domain=-1.3:1.3] plot ({-cosh(\x)}, {sinh(\x)});
        \draw[curve] (-0.5,0) circle (1.5);
        \draw[curve] (0,0) ellipse ({sqrt(6)} and {sqrt(2)});
\end{tikzpicture}
\end{minipage}
\begin{minipage}[t]{0.5\textwidth}
\centering
\begin{tikzpicture}[scale=0.9]
        \draw[curve] (0,0) ellipse ({sqrt(3)} and 1);
        \draw[curve] ({sqrt(3)/6}, {-sqrt(3)/2}) circle ({sqrt(10/3)});
        \draw[curve, domain=-1.2:0.9, samples=100] plot ({sqrt(2)*cosh(\x)}, {sqrt(2)*sinh(\x)});
        \draw[curve, domain=-1.2:0.9, samples=100] plot ({-sqrt(2)*cosh(\x)}, {sqrt(2)*sinh(\x)});
\end{tikzpicture}
\end{minipage}
\caption{Geometric realizations of arrangements of combinatorial type $W^{3,0}_{8,1}$ and $W^{3,0}_{8,2}$.}
\end{figure}

\newpage
\begin{proposition}\label{prop:strong-ziegler-pairs}
The pairs \[W^{3,0}_{5,1},W^{3,0}_{5,2},\] \[W^{3,0}_{8,1},W^{3,0}_{8,2},\]  and \[W^{2,2}_{4,1}, W^{2,2}_{4,2}\] form strong Ziegler pairs.
\end{proposition}

\begin{proof}
We first record the homological difference. A \texttt{SINGULAR} computation gives the following minimal graded free resolutions of the corresponding Milnor algebras:
\[
\left(W^{3,0}_{5,1}\right):
\quad
0\longrightarrow S^2(-11)
\longrightarrow S^2(-10)\oplus S(-9)\oplus S(-8)
\longrightarrow S^3(-5)
\longrightarrow S,
\]
\[
\left(W^{3,0}_{5,2}\right):
\quad
0\longrightarrow S(-11)\oplus S(-10)
\longrightarrow S^4(-9)
\longrightarrow S^3(-5)
\longrightarrow S,
\]
\[
\left(W^{3,0}_{8,1}\right):
\quad
0\longrightarrow S(-11)
\longrightarrow S^2(-9)\oplus S(-8)
\longrightarrow S^3(-5)
\longrightarrow S,
\]
\[
\left(W^{3,0}_{8,2}\right):
\quad
0\longrightarrow S^3(-10)
\longrightarrow S^5(-9)
\longrightarrow S^3(-5)
\longrightarrow S,
\]
and
\[
\left(W^{2,2}_{4,1}\right):
\quad
0\longrightarrow S(-11)
\longrightarrow S^2(-9)\oplus S(-8)
\longrightarrow S^3(-5)
\longrightarrow S,
\]
\[
\left(W^{2,2}_{4,2}\right):
\quad
0\longrightarrow S(-11)\oplus S(-10)
\longrightarrow S(-10)\oplus S^2(-9)\oplus S(-8)
\longrightarrow S^3(-5)
\longrightarrow S.
\]
Thus the resolutions are different within each of the three listed pairs.

It remains to justify that the strong combinatorics are the same within each pair. For the pairs
\[
W^{3,0}_{5,1},\quad W^{3,0}_{5,2}
\]
and
\[
W^{3,0}_{8,1},\quad W^{3,0}_{8,2},
\]
this is checked by the incidence description in Remark~5.9. For the pair
\[
W^{2,2}_{4,1},\quad W^{2,2}_{4,2},
\]
the Appendix gives the same weak combinatorics
\[
(2,2;7,0,2)
\]
and the same incidence vector
\[
\mathbf{r}=(0,0,1,0,1,2,0,0).
\]
Moreover, the incidence pattern is the same: in both arrangements the two ordinary triple points are incident with the two conics and with one of the two lines, while the remaining line meets the other components only in nodes. Hence there is a bijection of components and a bijection of singular points preserving degrees, topological types, and incidences.

Therefore all three listed pairs have the same strong combinatorial type but different Jacobian syzygy modules. Hence they form strong Ziegler pairs.
\end{proof}

\begin{remark}
For the pair
\[
W^{3,0}_{5,1},\quad W^{3,0}_{5,2},
\]
let the irreducible conics be indexed by
\[
\mathbf{i}=\{1,2,3\}.
\]
In both arrangements the two ordinary triple points are incident with all three components, while the six nodes are distributed in the same way among the three unordered pairs of components. Therefore there is a bijection of components and a bijection of singular points preserving the topological type of each singular point and the incidence relation. Hence
\[
\mathcal{W}\left(W^{3,0}_{5,1}\right)=\mathcal{W}\left(W^{3,0}_{5,2}\right).
\]
The same verification applies to the pair
\[
W^{3,0}_{8,1},\quad W^{3,0}_{8,2}.\]

For the pair
\[
W^{2,2}_{4,1},\quad W^{2,2}_{4,2},
\]
we index the components as
\[
Q_1,Q_2,L_1,L_2,
\]
where \(Q_1,Q_2\) are the conics. In both arrangements the two ordinary triple points are incident with
\[
Q_1,\ Q_2,\ L_1,
\]
while the line \(L_2\) does not pass through any ordinary triple point. The remaining singularities are seven nodes, distributed in the same way among the pairs of components. Equivalently, both arrangements have the same incidence vector
\[
\mathbf{r}=(0,0,1,0,1,2,0,0).
\]
Thus, after a suitable relabelling of the two lines, the full strong combinatorial data coincide:
\[
\mathcal{W}\left(W^{2,2}_{4,1}\right)=\mathcal{W}\left(W^{2,2}_{4,2}\right).\]
\end{remark}

 \section*{Acknowledgments}
The author would like to thank Piotr Pokora for his guidance and for many valuable remarks and suggestions that improved this paper.

Artur Bromboszcz is supported by the National Science Centre (Poland) Sonata Bis Grant {\bf 2023/50/E/ ST1/00025}. For the purpose of Open Access, the author has applied a CC-BY public copyright license to any Author Accepted Manuscript (AAM) version arising from this submission.

\vskip 0.5 cm

\bigskip
Artur Bromboszcz,
Department of Mathematics,
University of the National Education Commission Krakow,
Podchor\c a\.zych 2,
PL-30-084 Krak\'ow, Poland. \\
\nopagebreak
\textit{E-mail address:} \texttt{artur.bromboszcz@uken.krakow.pl}

\section*{Appendix}
\label{sec:data}

The table below records conic-line arrangements of degree at most \(6\),
together with the combinatorial data and algebraic invariants considered in this article. Here \(W=(k,\ell;n_2,t_3,n_3)\) is a vector of weak combinatorics from Definition~\ref{def:WeakComb}, where \(k\) is the number of conics, \(\ell\) is the number of lines, \(n_2\) is the number of nodes, \(t_3\) is the number of tacnodes, and \(n_3\) is the number of ordinary triple points. The \(\mathbf r\)-vector of incidences encodes the combinatorial data of the arrangement, see Definition~\ref{def:StrongComb}.
The tuple \((d;d_1,\ldots,d_s)\) gives the degree of a defining polynomial of an \(s\)-syzygy curve \(\{\mathcal{C}:{f=0}\}\subset\mathbb P^2_{\mathbb C}\) and the corresponding exponents \(d_i\),  \((i=1,...,s)\).
Moreover, \(\tau(\mathcal C)\) denotes the total Tjurina number of a curve $\mathcal{C}$, see equation \eqref{eq:Trurina} in Preliminaries. The last column gives the type of the curve.
In the column \(\tau(C)\), underlined entries indicate arrangements forming weak Ziegler pair or triple, whereas circled entries indicate arrangements forming strong Ziegler pair.

\begin{landscape}

\begin{table}[h!]
 \centering
 \renewcommand{\arraystretch}{1.5}
 \begin{tabular}{@{}llcclc@{}}
 \toprule 
 \textbf{Combinatorial type} & \textbf{$\mathbf{r}$-vector} & \textbf{Exponents} & \textbf{Tjurina} & \textbf{Defining polynomial $f$} & \multirow{2}{*}{\textbf{Type}} \\ $W=(k,\ell;n_2,t_3,n_3)$ & $(r_1,r_2,\ldots,r_{4(k-1)+2\ell})$ & $(d;d_1,d_2,\ldots,d_s)$ & $\tau{(\mathcal{C})}$& $f\in\mathbb{C}[x,y,z]$ \, \, (ex.) &  \\
 \midrule
 \midrule
 $W^{1,1}_1=(1,1;2,0,0)$ & $(0,2)$ & $(3; 1,2,2)$ & 2 & $(x^2+y^2+z^2)x$ & Nearly free \\

 $W_2^{1,1}=(1,1;0,1,0)$ & $(2,0)$ & $(3;1,1)$ & 3 & $(x^2-y^2+z^2)(x-y)$  & Free \\
 
 \midrule
 $W_1^{2,0}=(2,0;4,0,0)$ & $(0,0,0,2)$ & $(4;2,3,3,3)$ & 4 & $(x^2+4y^2-4z^2)(4x^2+y^2-4z^2)$ & 2B \\

 $W_2^{2,0}=(2,0;2,1,0)$ & $(0,0,2,0)$ & $(4;2,2,3)$ & 5 & $(x^2-y^2-z^2)(x^2-2xz+y^2-3z^2)$ & MPOG \\

 $W_3^{2,0}=(2,0;0,2,0)$ & $(0,2,0,0)$ & $(4;1,3,3)$ & 6 & $(x^2+y^2-z^2)(x^2+y^2+z^2)$ & Nearly free \\
 
 \midrule
 $W_1^{1,2}=(1,2;5,0,0)$ & $(0,0,2,1)$ & $(4; 2,2,3)$ & 5 & $(x^2+y^2-z^2)(x-y)(x+y)$ & MPOG \\
 
 $W_2^{1,2}=(1,2;3,1,0)$ & $(0,0,1,2)$ & $(4;2,2,2)$ & 6 & $(x^2-y^2-z^2)(x-y)x$ & Nearly free \\

 $W_3^{1,2}=(1,2;2,0,1)$ & $(0,2,1,0)$ & $(4; 2,2,2) $ & 6 & $xy(xy+xz+yz)$ & Nearly free \\
 
$W_4^{1,2}=(1,2;1,2,0)$ & $(0,3,0,0)$ & $(4;1,2)$ & 7 & $(x^2-y^2-z^2)(x-y)(x+y)$ & Free \\

\midrule
 $W_1^{2,1}=(2,1;8,0,0)$ & $(0,0,0,1,0,2)$ & $(5; 3,3,4,4)$ & 8 & $(x^2-y^2-z^2)(x^2+y^2-2z^2)x$ & 2B \\
\hdashline

 \multirow{2}{*}{$W_2^{2,1}=(2,1;6,1,0)$} &
$(0,0,1,0,1,1)$  & \multirow{2}{*}{$(5;3,3,3)$} & \multirow{2}{*}{9} & $(x^2-y^2-z^2)(x^2+y^2-2z^2)(x-z)$ & \multirow{2}{*}{2A} \\

 &
$(0,0,0,1,2,0)$ & & & $(x^2-y^2-z^2)(x^2+y^2+2xz-3z^2)x$  & \\
\hdashline

  $W_3^{2,1}=(2,1;5,0,1)$ & $(0,0,1,0,2,0)$ & $(5;3,3,3)$ & 9 & $(x^2+3y^2-4z^2)(3x^2+y^2-4z^2)(2x+y+z)$ & 2A \\

\hdashline

 $W_{4,1}^{2,1}=(2,1;4,2,0)$ & $(0,0,0,3,0,0)$ & $(5;2,3,4)$ & \multirow{3}{*}{\underline{ $\bf 10$ }} & $(x^2-y^2-z^2)(x^2+y^2-z^2)x$  & MPOG \\
 
 \multirow{2}{*}{$W_{4,2}^{2,1}=(2,1;4,2,0)$} & $(0,0,1,1,1,0)$ & \multirow{2}{*}{$(5;3,3,3,3)$} &  & $(x^2+3y^2-3z^2)(x^2+y^2+2yz-3z^2)(x+y+2z)$ & \multirow{2}{*}{2B}  \\

& $(0,1,0,0,2,0)$ & &  & $(x^2+3y^2-3z^2)(3x^2+y^2-3z^2)(x-y+2z)$ &   \\

 \bottomrule
 \end{tabular}
 \caption*{Selected conic-line arrangements of degree at most \(6\) and their combinatorial and algebraic invariants.}
 \end{table}
\end{landscape}

\begin{landscape}
\begin{table}[h]
 \centering
 \renewcommand{\arraystretch}{1.5}
 \begin{tabular}{@{}llcclc@{}}
 \textbf{Combinatorial type} & \textbf{$\mathbf{r}$-vector} & \textbf{Exponents} & \textbf{Tjurina} & \textbf{Defining polynomial $f$} & \multirow{2}{*}{\textbf{Type}} \\ $W=(k,\ell;n_2,t_3,n_3)$ & $(r_1,r_2,\ldots,r_{4(k-1)+2\ell})$ & $(d;d_1,d_2,\ldots,d_s)$ & $\tau{(\mathcal{C})}$& $f\in\mathbb{C}[x,y,z]$ \, \, (ex.) &  \\
 \midrule
 \midrule
  $W_5^{2,1}=(2,1;0,4,0)$ & \multicolumn{5}{l}{\makecell[l]{This case is not geometrically realizable.}}  \\

  $W_6^{2,1}=(2,1;2,0,2)$ & $(0,1,0,2,0,0)$ & $(5;2,3,4)$ & 10 & $(x^2+3y^2-4z^2)(3x^2+y^2-4z^2)(x+y)$ & MPOG \\

 $W_7^{2,1}=(2,1;3,1,1)$ & $(0,0,1,2,0,0)$ & $(5;3,3,3,3)$ & 10 & $(x^2-y^2-z^2)(x^2+y^2-2xz-3z^2)(y+\sqrt{3}z)$ & 2B \\
\hdashline
     \multirow{2}{*}{$W_8^{2,1}=(2,1;2,3,0)$} & $(0,0,2,1,0,0)$ & \multirow{2}{*}{$(5;2,3,3)$} & \multirow{2}{*}{11} & $(x^2+y^2-z^2)(x^2-y^2-z^2)(y-z)$ & \multirow{2}{*}{Nearly free}  \\

      & $(0,1,0,2,0,0)$ &  &  & $(x^2+4y^2-4z^2)(x^2+y^2+2yz-3z^2)(x+2z)$ &  \\

      \hdashline

  $W_9^{2,1}=(2,1;0,1,2)$ & $(0,1,2,0,0,0)$ & $(5;2,3,3)$ & 11 & $(x^2-y^2-z^2)(x^2+y^2-2xz-3z^2)(x-2z)$ & Nearly free \\

    $W_{10}^{2,1}=(2,1;1,2,1)$ & \multicolumn{5}{l}{\makecell[l]{This case is not geometrically realizable by Lemma~\ref{21121}.}}  \\

 \midrule
 $W_1^{1,3}=(1,3;9,0,0)$ & $(0,0,0,3,0,1)$ & $(5;3,3,3)$ & 9 & $(x^2+y^2-3z^2)(x-y)(x+y)z$ & 2A \\
 
  $W_2^{1,3}=(1,3;7,1,0)$ & $(0,0,1,2,1,0)$ & $(5;3,3,3,3)$ & 10 & $(x^2-y^2-z^2)(x-y)(y-z)y$ & 2B \\

\hdashline
    $W_{3,1}^{1,3}=(1,3;6,0,1)$ & $(0,0,2,1,1,0)$ & $(5;3,3,3,3)$ & \multirow{2}{*}{\underline{ $\bf 10$ }} & $(x^2+y^2-z^2)(x+y+z)(x-y+z)z$ & 2B \\
    
    $W_{3,2}^{1,3}=(1,3;6,0,1)$ & $(0,0,3,0,0,1)$ & $(5;2,3,4)$ & & $(x^2+y^2-z^2)(x+y)xy$ & MPOG \\
\hdashline
 
   $W_4^{1,3}=(1,3;5,2,0)$ & $(0,0,2,2,0,0)$ & $(5;2,3,3)$ & 11 & $(x^2+y^2-z^2)(x-z)(y-z)(x+y)$ & Nearly free \\

    $W_5^{1,3}=(1,3;3,0,2)$ & $(0,1,2,1,0,0)$ & $(5;2,3,3)$ & 11 & $(x^2+y^2-2z^2)(x+z)(x-z)(y+z)$ & Nearly free \\

\hdashline
    \multirow{2}{*}{$W_6^{1,3}=(1,3;4,1,1)$} & $(0,1,2,0,1,0)$ & \multirow{2}{*}{$(5;2,3,3)$} & \multirow{2}{*}{11} & $(x^2+y^2- 2z^2)(x+y)(x+y+z)(x+y+2z)$ & \multirow{2}{*}{Nearly free} \\

     & $(0,0,3,1,0,0)$ & & & $(x^2+y^2-2z^2)(x+y)(y-z)(x+y+2z)$ & \\

\hdashline

  $W_7^{1,3}=(1,3;3,3,0)$ & $(0,0,4,0,0,0)$ & $(5;2,2)$ & 12 & $(x^2+3y^2-3z^2)(x-y-2z)(x-y+2z)(y-z)$ & Free \\

      $W_8^{1,3}=(1,3;0,0,3)$ & $(0,3,1,0,0,0)$ & $(5;2,2)$ & 12 & $(x^2+y^2-z^2)(x+y+z)(x-y+z)x$ & Free \\
    
    $W_9^{1,3}=(1,3;2,2,1)$ & $(0,2,1,1,0,0)$ & $(5;2,2)$ & 12 & $(x^2+y^2-z^2)(x-y)(x+z)(y+z)$ & Free \\

  \midrule
  
   $W_1^{3,0}=(3,0;12,0,0)$ & $(0,0,0,0,0,0,0,3)$ & $(6;4,4,5,5,5)$ & 12 & $ (x^2+3y^2-3z^2)(3x^2+y^2-3z^2)(x^2+y^2-2z^2)$ & 3C \\
  
  $W_2^{3,0}=(3,0;10,1,0)$ & $(0,0,0,0,0,0,2,1)$ & $(6;4,4,4,5)$ & 13 &
$(x^2 + y^2 + yz - 2z^2)(x^2 + 3y^2 - 3z^2)(x^2 + y^2 - 2z^2)$ & 3B \\

  $W_3^{3,0}=(3,0;9,0,1)$ & $(0,0,0,0,0,0,3,0)$ & $(6;4,4,4,5)$ & 13 & $(x^2+yz-z^2)(x^2+2y^2-z^2)(x^2+y^2+xz+yz)$ & 3B \\
  
  $W_4^{3,0}=(3,0;8,2,0)$ & $(0,0,0,0,2,1,0,0)$ & $(6;4,4,4,4)$ & 14 & $(x^2+2y^2-xz)(2x^2+y^2-yz)(x^2+y^2-z^2)$ & 3B' \\

 \bottomrule

 \end{tabular}

 \end{table}
 \end{landscape}

 \begin{landscape}
\begin{table}[h]
 \centering
 \renewcommand{\arraystretch}{1.5}
 \begin{tabular}{@{}llcclc@{}}
 \textbf{Combinatorial type} & \textbf{$\mathbf{r}$-vector} & \textbf{Exponents} & \textbf{Tjurina} & \textbf{Defining polynomial $f$} & \multirow{2}{*}{\textbf{Type}} \\ $W=(k,\ell;n_2,t_3,n_3)$ & $(r_1,r_2,\ldots,r_{4(k-1)+2\ell})$ & $(d;d_1,d_2,\ldots,d_s)$ & $\tau{(\mathcal{C})}$& $f\in\mathbb{C}[x,y,z]$ \, \, (ex.) &  \\
 \midrule
 \midrule

          $W_{5,1}^{3,0}=(3,0;6,0,2)$ & \multirow{2}{*}{\textbf{$(0,0,0,0,0,3,0,0)$}} & $(6;3,4,5,5)$ & \multirow{2}{*}{\textbf{{\Circled{ 14 }}}} & $(x^2+2y^2-3z^2)(x^2+y^2-2z^2)(x^2+yz-2z^2)$ & 2B \\

  $W_{5,2}^{3,0}=(3,0;6,0,2)$ & & $(6;4,4,4,4)$ & & $(x^2+2y^2-3z^2)(x^2+y^2-2z^2)(x^2+2xz+2yz-z^2)$ & 3B' \\
\hdashline

    $W_{6}^{3,0}=(3,0;7,1,1)$ & $(0,0,0,0,0,2,1,0)$ & $(6;4,4,4,4)$ & 14 & $(x^2+yz)(x^2+xz+2yz)(x^2+y^2+yz+xz)$ & 3B' \\

    \hdashline
  $W_{7,1}^{3,0}=(3,0;6,3,0)$ & $(0,0,0,0,1,1,1,0)$ & $(6;3,4,4)$ & \multirow{2}{*}{{\underline{ $\bf 15$ }}} & $(4x^2+y^2-4z^2)(x^2-y^2-z^2)(x^2+y^2-2yz)$ & 2A \\

  $W_{7,2}^{3,0}=(3,0;6,3,0)$ & $(0,0,0,0,0,3,0,0)$ & $(6;4,4,4,4,4)$ &  & $(x^2+y^2-2xz)(x^2+y^2+2xz)(x^2-y^2-4z^2)$ & 3C \\

\hdashline
    $W_{8,1}^{3,0}=(3,0;4,1,2)$ & \multirow{2}{*}{$(0,0,0,0,2,1,0,0)$} & $(6;3,4,4)$ & \multirow{2}{*}{\textbf{{\Circled{ 15 }}}} & $(x^2-y^2-z^2)(x^2+y^2+xz-2z^2)(x^2+3y^2-6z^2)$ & 2A \\

    $W_{8,2}^{3,0}=(3,0;4,1,2)$ & & $(6; 4,4,4,4)$ & & $(x^2+3y^2-3z^2)(6x^2+6y^2-2\sqrt{3}xz+6\sqrt{3}yz-15z^2)(x^2-y^2-2z^2)$ & 3B' \\
\hdashline

    $W_{9}^{3,0}=(3,0;5,2,1)$ & $(0,0,0,0,1,2,0,0)$ & $(6;4,4,4,4)$ & 15 & $(x^2+3y^2-4z^2)(x^2+y^2+2xz)(x^2-y^2-xy+xz)$ & 3B' \\

        $W_{10}^{3,0}=(3,0;3,0,3)$ & $(0,0,0,0,3,0,0,0)$ & $(6;4,4,4,4,4)$ & 15 & $(x^2+2y^2-3z^2)(x^2+y^2-2z^2)(2x^2-2y^2+z^2+xy-xz-yz)$ & 3C \\
        
        $W_{11}^{3,0}=(3,0;0,0,4)$ & $(0,0,0,3,0,0,0,0)$ & $(6;2,5,5,5)$ & 16 & $(x^2+2y^2-3z^2)(x^2+y^2-2z^2)(2x^2+y^2-3z^2)$ & 2B \\

    $W_{12}^{3,0}=(3,0;2,2,2)$ & $(0,0,0,1,2,0,0,0)$ & $(6;3,4,4,4)$ & 16 & $(x^2+yz)(x^2+2yz-xz)(x^2+4y^2+xz+4yz)$ & 2B \\

    $W_{13}^{3,0}=(3,0;3,3,1)$ & $(0,0,0,0,3,0,0,0)$ & $(6;3,4,4,4)$ & 16 & \small{$(x^2+3y^2-3z^2)(x^2+y^2-yz-2z^2)$}\footnotesize{$(25x^2-6xy+33y^2+12xz-36yz-60z^2)$} & 2B \\

\hdashline
              $W_{14,1}^{3,0}=(3,0;4,4,0)$ & $(0,0,0,1,0,2,0,0)$ & $(6;3,3,5)$ & \multirow{2}{*}{\underline{ $\bf 16$ }} & $(x^2+y^2-z^2)(4x^2+y^2-4z^2)(x^2+4y^2-4z^2)$ & POG \\

    $W_{14,2}^{3,0}=(3,0;4,4,0)$ & $(0,0,0,0,1,2,0,0)$ & $(6;3,4,4,4)$ &  & $(x^2+y^2-z^2)(x^2+4y^2-4z^2)(x^2-y^2+xz-2z^2)$ & 2B \\
\hdashline

    $W_{15}^{3,0}=(3,0;0,3,2)$ & $(0,0,0,3,0,0,0,0)$ & $(6;3,3,4)$ & 17 & $(x^2+3y^2-4z^2)(x^2+y^2+2xz)(x^2-3y^2-2xz)$ & MPOG \\

       $W_{16}^{3,0}=(3,0;2,5,0)$ & $(0,0,0,1,2,0,0,0)$
   & $(6;3,3,4)$ & 17 & $(x^2+y^2-2xz)(x^2+y^2+2xz)(2x^2+6y^2-9z^2)$ & MPOG \\
  
  $W_{17}^{3,0}=(3,0;0,6,0)$ & $(0,0,0,3,0,0,0,0)$ & $(6;3,3,3)$ & 18 & $(x^2-y^2-z^2)(x^2-y^2+z^2)(x^2+y^2-z^2)$ & Nearly free \\

    $W_{18}^{3,0}=(3,0;1,4,1)$ & \multicolumn{5}{l}{\makecell[l]{This case is not geometrically realizable by Proposition~\ref{prop:oneD4}.}} \\

\midrule
    
      $W_1^{2,2}=(2,2;13,0,0)$ & $(0,0,0,0,2,0,0,2)$ & $(6;4,4,4,5)$ & 13 & $(x^2-y^2-z^2)(x^2+y^2-2z^2)xy$ & 3B \\
      
  \hdashline
  \multirow{2}{*}{$W_2^{2,2}=(2,2;11,1,0)$} & $(0,0,0,1,1,0,1,1)$  & $(6;4,4,4,4)$ & 14 & $x(x+z)(4x^2+y^2-4z^2)(x^2+4y^2-4z^2)$ & \multirow{2}{*}{3B'} \\

    & $(0,0,0,0,2,0,2,0)$ & $(6;4,4,4,4)$ & 14 & $y(x-y)(4x^2+y^2+2yz-3z^2)(x^2+4y^2-4z^2)$  &  \\

    \bottomrule

 \end{tabular}

 \end{table}
 \end{landscape}

  \begin{landscape}
\begin{table}[h]
 \centering
 \renewcommand{\arraystretch}{1.5}
 \begin{tabular}{@{}llcclc@{}}
 \textbf{Combinatorial type} & \textbf{$\mathbf{r}$-vector} & \textbf{Exponents} & \textbf{Tjurina} & \textbf{Defining polynomial $f$} & \multirow{2}{*}{\textbf{Type}} \\ $W=(k,\ell;n_2,t_3,n_3)$ & $(r_1,r_2,\ldots,r_{4(k-1)+2\ell})$ & $(d;d_1,d_2,\ldots,d_s)$ & $\tau{(\mathcal{C})}$& $f\in\mathbb{C}[x,y,z]$ \, \, (ex.) &  \\
 \midrule
\midrule

  \multirow{2}{*}{$W_3^{2,2}=(2,2;10,0,1)$} & $(0,0,0,1,1,0,2,0)$ & \multirow{2}{*}{$(6;4,4,4,4)$} & \multirow{2}{*}{14} & $(x^2+3y^2-4z^2)(3x^2+y^2-4z^2)(x-2y+z)(x+2y)$ & \multirow{2}{*}{3B'} \\

  & $(0,0,2,0,0,0,1,1)$ & & & $(x^2+4y^2-4z^2)(4x^2+y^2-4z^2)(x+2y+z)y$ & \\

  \hdashline

  $W_{4,1}^{2,2}=(2,2;7,0,2)$ & \multirow{2}{*}{\Circled{ $(0,0,1,0,1,2,0,0)$ }} & $(6;3,4,4)$ & \multirow{3}{*}{\Circled{ $\bf 15$ }} & $(x^2+3y^2-4z^2)(3x^2+y^2-4z^2)(x-y+z)(x-y)$ & 2A \\

  $W_{4,2}^{2,2}=(2,2;7,0,2)$ & & $(6;3,4,4,5)$ & & $(x^2+3y^2-4z^2)(3x^2+y^2-4z^2)(x+z)x$ & 2B \\
  \cline{2-2}
\cline{5-5}

  $W_{4,3}^{2,2}=(2,2;7,0,2)$ & $(0,0,1,1,0,2,0,0)$ & $(6;4,4,4,4,4)$ &  & $(x^2+3y^2-4z^2)(3x^2+y^2-4z^2)(x+y+2z)y$ & 3C \\
\hdashline
          \multirow{3}{*}{$W_{5,1}^{2,2}=(2,2;9,2,0)$} & $(0,0,1,0,1,0,2,0)$ & \multirow{3}{*}{$(6;4,4,4,4,4)$} & \multirow{5}{*}{\underline{ $\bf 15$ }} & $(x^2+3y^2-3z^2)(3x^2+y^2-3z^2)(x-y-2z)x$ & \multirow{3}{*}{3C} \\

          & $(0,0,0,2,0,0,2,0)$ & & & $(x^2+3y^2-3z^2)(4x^2+y^2-4z^2)(x-y+2z)(x-z)$ & \\
          
          & $(0,0,0,1,1,1,1,0)$ & & & $(x^2+4y^2-4z^2)(x^2+y^2+2yz-3z^2)(y+z)y$ & \\

  $W_{5,2}^{2,2}=(2,2;9,2,0)$ & $(0,0,0,2,0,1,0,1)$ & $(6;3,4,4)$ & & $(x^2+3y^2-3z^2)(4x^2+y^2-4z^2)(x-y+2z)(y+z)$ & 2A \\

  $W_{5,3}^{2,2}=(2,2;9,2,0)$ & $(0,0,0,0,2,0,2,0)$ & $(6;3,4,4,5)$ & & $(x^2+4y^2-4z^2)(x^2+y^2-z^2)(x+y)y$ & 2B\\

  \hdashline

  \multirow{3}{*}{$W_{6}^{2,2}=(2,2;8,1,1)$} & $(0,0,0,2,0,1,1,0)$ & \multirow{3}{*}{$(6;4,4,4,4,4)$} & \multirow{3}{*}{15} & $(x^2-y^2-z^2)(x^2+y^2-xz-2z^2)(x+y-z)(x-y-z)$ & \multirow{3}{*}{3C} \\

  & $(0,0,0,1,1,2,0,0)$ & & & $(x^2-y^2-xz-2z^2)(x^2+y^2-3xz-4z^2)(y-2z)(x-y)$ & \\
  
  & $(0,0,1,1,0,0,2,0)$ & & & $(x^2+y^2-2z^2)(4x^2+y^2-xz-4z^2)(x+y+2z)(x-2y-4z)$ & \\

   \hdashline
  
  \multirow{4}{*}{$W_7^{2,2}=(2,2;7,3,0)$} & $(0,0,0,1,2,1,0,0)$ & \multirow{4}{*}{$(6;3,4,4,4)$} & \multirow{4}{*}{16} & $(x^2+4y^2-4z^2)(x^2+y^2-z^2)(x-z)z$ & \multirow{4}{*}{2B} \\

    & $(0,0,1,0,1,2,0,0)$ & & & $(x^2+4y^2-4z^2)(x^2+y^2+2yz-3z^2)(x+2z)(x+y)$ & \\

    & $(0,0,0,2,1,0,1,0)$ & & & $(x^2+4y^2-4z^2)(x^2+y^2+yz-2z^2)(x+2z)(x-2z)$ & \\

    & $(0,0,1,1,0,1,1,0)$ & & & $(x^2+3y^2-3z^2)(3x^2+y^2-3z^2)(x+y+2z)(x+z)$ & \\
  
  \hdashline

   \multirow{6}{*}{$W_{8}^{2,2}=(2,2;5,1,2)$} & $(0,0,2,0,0,2,0,0)$ & \multirow{6}{*}{$(6;3,4,4,4)$} & \multirow{6}{*}{16} & $(x^2+y^2-2z^2)(4x^2+y^2-xz-4z^2)(x+y+2z)(3x-y-2z)$ & \multirow{6}{*}{2B} \\
   \cline{2-2}
\cline{5-5}

   & \multirow{2}{*}{{$\bf (0,0,2,0,1,1,0,0)$}} & & & {$(x^2+y^2-2z^2)(4x^2+y^2-xz-4z^2)(x+y+2z)(x-z)$} & \\

   & & & & {$(x^2+2y^2+6yz)(2x^2+y^2+6yz)(x+2y+6z)(x+y+2z)$} & \\
   \cline{2-2}
\cline{5-5}

   & $(0,0,1,0,3,0,0,0)$ & & & $(x^2+3y^2-12z^2)(x^2+3yz-6z^2)(y+z)y$ & \\

   & $(0,0,0,2,2,0,0,0)$ & & & $(x^2+3y^2-12z^2)(x^2+3yz-6z^2)(x+2y+5z)(x-2y-5z)$ & \\

   & $(0,0,1,1,1,1,0,0)$ & & & $(x^2+2y^2+6yz)(2x^2+y^2+6yz)(x+2y+6z)(x-y-3z)$ & \\

          \bottomrule

 \end{tabular}

 \end{table}
 \end{landscape}

\begin{landscape}
\begin{table}[h]
\centering
\renewcommand{\arraystretch}{1.5}
 \begin{tabular}{@{}llcclc@{}}
 \textbf{Combinatorial type} & \textbf{$\mathbf{r}$-vector} & \textbf{Exponents} & \textbf{Tjurina} & \textbf{Defining polynomial $f$} & \multirow{2}{*}{\textbf{Type}} \\ $W=(k,\ell;n_2,t_3,n_3)$ & $(r_1,r_2,\ldots,r_{4(k-1)+2\ell})$ & $(d;d_1,d_2,\ldots,d_s)$ & $\tau{(\mathcal{C})}$& $f\in\mathbb{C}[x,y,z]$ \, \, (ex.) &  \\
 \midrule
 \midrule

  \multirow{2}{*}{$W_{9}^{2,2}=(2,2;4,0,3)$} & $(0,0,1,1,2,0,0,0)$ & \multirow{2}{*}{$(6;3,4,4,4)$} & \multirow{2}{*}{16} &$(x^2+3y^2-4z^2)(3x^2+y^2-4z^2)(x-y)(x+2y-z)$ & \multirow{2}{*}{2B} \\

  & $(0,0,2,0,1,1,0,0)$ & & & $(x^2+3y^2-4z^2)(3x^2+y^2-4z^2)(x-y-2z)(x+y+2z)$ & \\
  
   \hdashline
   \multirow{4}{*}{$W_{10}^{2,2}=(2,2;6,2,1)$} & \multirow{2}{*}{$(0,0,0,2,1,1,0,0)$} & \multirow{4}{*}{$(6;3,4,4,4)$} & \multirow{4}{*}{16} & $(x^2+y^2-z^2)(x^2+2y^2-2z^2)(x+2y+z)y$ & \multirow{4}{*}{2B} \\

   & & & & $(x^2+3y^2-12z^2)(x^2+3yz-6z^2)(y+2z)(x+3y)$ & \\
   \cline{2-2}
\cline{5-5}

   & $(0,0,1,1,0,2,0,0)$ & & & $(x^2+3y^2-12z^2)(x^2+3yz-6z^2)(y+2z)\left(x+y+(2-2\sqrt{3})z\right)$ & \\

   & $(0,0,2,0,0,1,1,0)$ & & & $(4x^2+y^2-4z^2)(x^2+3y^2-3z^2)(x+y+2z)(x-y-2z)$ & \\

   \hdashline
  
      \multirow{3}{*}{$W_{11}^{2,2}=(2,2;5,4,0)$} & $(0,0,0,3,0,1,0,0)$ & \multirow{3}{*}{$(6;3,3,4)$} & \multirow{3}{*}{17} & $(x^2+4y^2-4z^2)(x^2+y^2-z^2)(x-z)(x+z)$ & \multirow{3}{*}{MPOG} \\

 & $(0,0,2,0,0,2,0,0)$ & & & $(x^2+3y^2-3z^2)(3x^2+y^2-3z^2)(x+y+2z)(x+y-2z)$ & \\
 
 & $(0,0,1,1,1,1,0,0)$ & & & $(x^2+4y^2-4z^2)(x^2+y^2+2yz-3z^2)(x+2z)(y+3z)$ & \\
      
      \hdashline

        $W_{12}^{2,2}=(2,2;1,0,4)$ & $(0,0,2,2,0,0,0,0)$ & $(6;2,4,5)$ & 17 & $(x^2-y^2-z^2)(x^2-y^2-yz)(y-z)z$ & MPOG \\
          $W_{13}^{2,2}=(2,2;2,1,3)$ & $(0,0,2,2,0,0,0,0)$ & $(6;3,3,4)$ & 17 & $(x^2+y^2-4xz-21z^2)(x^2-y^2-9z^2)(4x-3y-8z)(y+4z)$ & MPOG \\

\hdashline
  \multirow{3}{*}{$W_{14}^{2,2}=(2,2;3,2,2)$} & \multirow{2}{*}{$(0,0,2,0,2,0,0,0)$} & \multirow{3}{*}{$(6;3,3,4)$} & \multirow{3}{*}{17} & $(x^2+3y^2-3z^2)(3x^2+y^2-3z^2)(x+y+2z)(x-y)$ & \multirow{3}{*}{MPOG} \\

& & & & $(x^2+3y^2-4z^2)(2x^2-3y^2+3xz-2z^2)(x+3y+4z)(3x+y-4z)$ & \\
\cline{2-2}
\cline{5-5}

& $(0,0,1,2,1,0,0,0)$ & & & $(x^2+3y^2-4z^2)(2x^2-3y^2+3xz-2z^2)(x+3y+4z)(x-z)$ & \\
  
  \hdashline

    \multirow{2}{*}{$W_{15}^{2,2}=(2,2;4,3,1)$} &$(0,0,2,0,1,1,0,0)$ & \multirow{2}{*}{$(6;3,3,4)$} & \multirow{2}{*}{17} & $(2x^2+y^2+yz-2z^2)(x^2+3y^2-3z^2)(x+y+2z)(x-y-2z)$ & \multirow{2}{*}{MPOG}\\

& $(0,0,1,1,2,0,0,0)$ & & & $(x^2+y^2-z^2)(x^2-3y^2-z^2)(y-z)(x-2y)$ & \\

    \hdashline

  $W_{16}^{2,2}=(2,2;3,5,0)$ & $(0,0,2,0,2,0,0,0)$ & $(6;3,3,3)$ & 18 & $(x^2+4y^2-4z^2)(x^2+y^2+2yz-3z^2)(x-2z)(x+2z)$ & Nearly free\\
  \hdashline

  $W_{17}^{2,2}=(2,2;0,2,3)$ & \multicolumn{5}{l}{\makecell[l]{This case is not geometrically realizable by the B\'ezout count used in Lemma~\ref{21121}.}} \\

  $W_{18}^{2,2}=(2,2;1,3,2)$ & $(0,0,2,2,0,0,0,0)$ & $(6;3,3,3)$ & 18 & $(x^2+4y^2-4z^2)(x^2+y^2+2yz-3z^2)(3y+z)(x+2z)$ & Nearly free \\

  $W_{19}^{2,2}=(2,2;2,4,1)$ & \((0,0,2,1,1,0,0,0)\) & \((6;3,3,3)\) & \(18\) & \((y^2-xz)(x^2-10xy+61xz-24y^2-60yz+36z^2)(x-2y+z)x\) & Nearly free \\

    $W_{20}^{2,2}=(2,2;1,6,0)$ & \multicolumn{5}{l}{\makecell[l]{This case is not geometrically realizable.}} \\

  $W_{21}^{2,2}=(2,2;0,5,1)$ & \multicolumn{5}{l}{\makecell[l]{This case is not geometrically realizable by the B\'ezout count used in Lemma~\ref{21121}.}} \\

 \bottomrule

 \end{tabular}

 \end{table}
 \end{landscape}

  \begin{landscape}
\begin{table}[h]
 \centering
 \renewcommand{\arraystretch}{1.5}
 \begin{tabular}{@{}llcclc@{}}
 \textbf{Combinatorial type} & \textbf{$\mathbf{r}$-vector} & \textbf{Exponents} & \textbf{Tjurina} & \textbf{Defining polynomial $f$} & \multirow{2}{*}{\textbf{Type}} \\ $W=(k,\ell;n_2,t_3,n_3)$ & $(r_1,r_2,\ldots,r_{4(k-1)+2\ell})$ & $(d;d_1,d_2,\ldots,d_s)$ & $\tau{(\mathcal{C})}$& $f\in\mathbb{C}[x,y,z]$ \, \, (ex.) &  \\
 \midrule
 \midrule

     $W_{1}^{1,4}=(1,4;14,0,0)$ & $(0,0,0,0,4,0,0,1)$ & $(6;4,4,4,4)$ & 14 & $(x^2+y^2-z^2)(x-y)(x+y)(2x-z)(2x+z)$ & 3B' \\

    $W_{2}^{1,4}=(1,4;12,1,0)$ & $(0,0,0,1,3,0,1,0)$ & $(6;4,4,4,4,4)$ & 15 & $(x^2+y^2-z^2)(y-z)(x+y)(2x+z)y$ & 3C \\

\hdashline
            $W_{3,1}^{1,4}=(1,4;11,0,1)$ & $(0,0,0,3,1,0,0,1)$ & $(6;3,4,4)$ & \multirow{2}{*}{\underline{ $\bf 15$ }} & $(x^2+y^2-z^2)(x+y)xyz$ & 2A \\

    $W_{3,2}^{1,4}=(1,4;11,0,1)$ & $(0,0,0,2,2,0,1,0)$ & $(6;4,4,4,4,4)$ & & $(x^2+y^2-z^2)(x+y-z)(x-y+z)(x-y)(x+y)$ & 3C \\
\hdashline

    $W_{4}^{1,4}=(1,4;10,2,0)$ & $(0,0,0,2,2,1,0,0)$ & $(6;3,4,4,4)$ & 16 & $(x^2+y^2-z^2)(y-z)(2x+y)(x-y)(x+z)$ & 2B \\

    \hdashline

     \multirow{2}{*}{$W_{5}^{1,4}=(1,4;9,1,1)$} & $(0,0,1,2,1,0,1,0)$ & \multirow{2}{*}{$(6;3,4,4,4)$} &  \multirow{2}{*}{16} & $(x^2+y^2-z^2)(x+z)(2x-z)(x-y)x$ &  \multirow{2}{*}{2B} \\

     & $(0,0,0,3,1,1,0,0)$ & & & $(x^2+y^2-z^2)(x+z)(2x+y+z)(x-y)x$ & \\

     \hdashline

    \multirow{2}{*}{$W_{6}^{1,4}=(1,4;8,0,2)$} & $(0,0,1,3,0,0,1,0)$  & \multirow{2}{*}{$(6;3,4,4,4)$} & \multirow{2}{*}{16} & $(x^2+y^2-z^2)x(x-y)(x+y)(2x+y+z)$ & \multirow{2}{*}{2B} \\

& $(0,0,1,2,1,1,0,0)$ & & & $(x^2+y^2-2z^2)(x-z)(x+z)(x-y)y$ & \\
    
    \hdashline

        \multirow{3}{*}{$W_{7}^{1,4}=(1,4;5,0,3)$} & $(0,0,2,2,1,0,0,0)$ &  \multirow{3}{*}{$(6;3,3,4)$} &  \multirow{3}{*}{17} & $(x^2+y^2-z^2)(x+y-z)(x-y+z)(x+2y+z)(x-2y-z)$ &  \multirow{3}{*}{MPOG} \\

        & $(0,0,3,0,2,0,0,0)$ & & & $(x^2+y^2-z^2)(x+y-z)(x-y+z)(2y-z)y$ & \\
        
        & $(0,0,3,1,0,1,0,0)$ & & & $(x^2+y^2-2z^2)(x-z)(x+z)(y+z)x$ & \\
        
        \hdashline

            $W_{8}^{1,4}=(1,4;8,3,0)$ & $(0,0,0,3,2,0,0,0)$ & $(6;3,3,4)$ & 17 & $(x^2+y^2-z^2)(y+z)(y-z)(x+z)(x-2y)$ & MPOG \\
            \hdashline

    \multirow{2}{*}{$W_{9}^{1,4}=(1,4;6,1,2)$} & $(0,0,2,2,0,1,0,0)$ &  \multirow{2}{*}{$(6;3,3,4)$} &  \multirow{2}{*}{17} & $(x^2+y^2-2z^2)(x-y+2z)(3x+y-2z)(y+z)x$ &  \multirow{2}{*}{MPOG} \\

    & $(0,0,1,3,1,0,0,0)$ & & & $(x^2+y^2-25z^2)(x+5z)(y-4z)(y+4z)(x-3z)$ & \\
    
    \hdashline

    \multirow{3}{*}{$W_{10}^{1,4}=(1,4;7,2,1)$} & $(0,0,2,1,1,1,0,0)$ &  \multirow{3}{*}{$(6;3,3,4)$} &  \multirow{3}{*}{17} & $(x^2+y^2-z^2)(x-z)(x+z)(2y-z)z$ &  \multirow{3}{*}{MPOG} \\

    & $(0,0,1,3,0,1,0,0)$ & & & $(x^2+y^2-z^2)(x+z)(y-z)(2x-z)z$ & \\
    & $(0,0,0,4,1,0,0,0)$ & & & $(x^2+y^2-z^2)(x+y-z)(x-y+z)(x+y-\sqrt{2}z)(x-y+\sqrt{2}z)$ & \\

    \hdashline

    $W_{11}^{1,4}=(1,4;6,4,0)$ & $(0,0,0,5,0,0,0,0)$ & $(6;3,3,3)$ & 18 & $(x^2+y^2-z^2)(x-z)(x+z)(y-z)(y+z)$ & Nearly free \\

    $W_{12}^{1,4}=(1,4;2,0,4)$ & $(0,0,4,1,0,0,0,0)$ & $(6;2,4,4)$ & 18 & $(x^2+y^2-2z^2)(x-z)(x+z)(y-z)(y+z)$ & Nearly free\\

    $W_{13}^{1,4}=(1,4;5,3,1)$ & $(0,0,2,2,1,0,0,0)$ & $(6;3,3,3)$ & 18 & $(x^2+y^2-2z^2)(x-y+2z)(x+y-2z)(x+y+2z)x$ & Nearly free \\

    $W_{14}^{1,4}=(1,4;4,2,2)$ & $(0,0,3,1,1,0,0,0)$ & $(6;3,3,3)$ & 18 & $(x^2+y^2-z^2)(x-z)(x+z)(2x-y-z)x$ & Nearly free \\

    $W_{15}^{1,4}=(1,4;2,3,2)$ & \multicolumn{5}{l}{\makecell[l]{This case is not geometrically realizable.}}  \\

  \bottomrule

 \end{tabular}
 \end{table}
 \end{landscape}


\begin{thebibliography}{000}

\bibitem{NewH} T. Abe, A. Dimca, P. Pokora, A new hierarchy for complex plane curves. {\it Canad. Math. Bull.} Advance Publication 1--24 (2025), \url{https://doi.org/10.4153/S0008439525101422}.

\bibitem{Arnold} V. I. Arnold, Local normal forms of functions. \textit{Invent. Math}. {\bf 35}: 87--109 (1976).

\bibitem{BJP} A. Bromboszcz, B. Jarosławski, P. Pokora, On plus-one generated arrangements of plane conics. {\it Geometriae Dedicata} {\bf 220}: Art. Id. 10 (2026).

\bibitem{Singular} W.~Decker, G.-M. Greuel, G.~Pfister, H.~Sch\"onemann, Singular {4-1-1} --- {A} computer algebra system for polynomial computations, \url{http://www.singular.uni-kl.de} (2018).

\bibitem{Sernesi}
A. Dimca, E. Sernesi, Syzygies and logarithmic vector fields along plane curves. (Syzygies et champs de vecteurs logarithmiques le long de courbes planes.) \textit{J. \'Ec. Polytech., Math.} \textbf{1}: 247--267 (2014).

\bibitem{typ3} A. Dimca, G. Sticlaru, On type three complex plane curves. \texttt{arXiv:2601.01824}, to appear in \textit{Bull. Math. Soc. Sci. Math. Roumanie}.

\bibitem{DS20} 
A. Dimca, G. Sticlaru, Plane curves with three syzygies, minimal Tjurina curves, and nearly cuspidal curves. \textit{Geom. Dedicata}  {\bf 207}:  29--49 (2020).

\bibitem{Cremona} S. H. Hassanzadeh, A. Simis, Plane Cremona maps: Saturation and regularity of the base ideal. \textit{J. Algebra} \textbf{371}: 620--652 (2012).

\bibitem{BJ}
B. Jarosławski, On a Poincaré-type polynomial for plus-one generated plane curves. \textit{Ann. Univ. Ferrara} \textbf{72}: Art. Id 15 (2026).

\bibitem{Megyesi} G. Megyesi, Configurations of conics with many tacnodes. {\it Tohoku Math. J.} {\bf 52(4)}: 555--577 (2000).

\bibitem{MP24} A. M\u{a}cinic and P. Pokora, On plus–one generated conic-line arrangements with simple singularities. {\it Atti Accad. Naz. Lincei, Cl. Sci. Fis. Mat. Nat., IX. Ser., Rend. Lincei, Mat. Appl.} {\bf 35(3)}: 349--364 (2024).

\bibitem{PokoraPoly} P. Pokora, On Poincar\'e polynomials for plane curves with quasi-homogeneous singularities. { \it Bull. Lond. Math. Soc.} {\bf 57(8)}: 2549--2560 (2025).

\bibitem{Qarr} P. Pokora, $\mathcal{Q}$--conic arrangements in the complex projective plane. {\it Proc. Amer. Math. Soc.}, {\bf 151(7)}: 2873--2880 (2023).

\end{thebibliography}
\end{document}